\documentclass[11pt]{article}

 \usepackage{dsfont}
 \usepackage{color}
 \usepackage{graphicx}

\usepackage{mathrsfs}
\usepackage{amsfonts}
\usepackage{amssymb}
\usepackage{amsmath}
 \allowdisplaybreaks[4]
 \usepackage{dsfont}
\newtheorem{theorem}{Theorem}

\newtheorem{lemma}[theorem]{Lemma}

\newtheorem{proposition}[theorem]{Proposition}
\newtheorem{remark}[theorem]{Remark}

\newenvironment{proof}[1][Proof]{\textbf{#1.}}{\ \rule{0.5em}{0.5em}}%

\allowdisplaybreaks[4]

\begin{document}
\parindent 9mm
\title{Output Feedback Exponential Stabilization for {\color{blue} a 1-d Wave PDE with Dynamic Boundary}
\thanks{This work was supported by the Natural Science Foundation of Shaanxi Province
2018JM1051, 2014JQ1017. }
}
\author{ Zhan-Dong Mei
\thanks{School of Mathematics and Statistics, Xi'an Jiaotong
University, Xi'an 710049, China.  Email: zhdmei@mail.xjtu.edu.cn. } }


\date{}
\maketitle \thispagestyle{empty}
\begin{abstract}
%
%
We study the output feedback exponential stabilization for a 1-d wave PDE with dynamic boundary.
With only one measurement, we construct an infinite-dimensional state observer to trace the state and design an estimated state based controller to exponentially stabilize the original system. This is an essentially important improvement for the existence literature [\"{O}. Morg\"{u}l, B.P. Rao and F. Conrad, IEEE Transactions on Automatic Control, 39(10) (1994), 2140-2145] where two measurements including the high order angular velocity feedback were adopted. When a control matched nonlinear internal uncertainty and external disturbance are taken into consideration, we construct an infinite-dimensional extended state observer (ESO) to estimate the total disturbance and state simultaneously. By compensating the total disturbance, an estimated state based controller is designed to exponentially stabilize the original system while making the closed-loop system bounded. Riesz basis approach is crucial to the verifications of the exponential stabilities of two coupled systems of the closed-loop systems. Some numerical simulations are presented to illustrate the effectiveness.
\vspace{0.5cm} 

%
%

\end{abstract}


\section{Introduction}\label{section1}

In this paper, we are concerned with the output feedback exponential stabilization of a 1-d wave PDE with dynamic boundary described as follows
\begin{equation} \label{beem}
\left\{\begin{array}{l}
u_{tt}(x,t)=u_{xx}(x,t),\;\; x\in (0,1), \; t>0, \\
u(0,t)=0,u_{x}(1,t)+mu_{tt}(1,t)=U(t)+F(t),  t\ge 0,\\
y(t)=\{u_{x}(0,t),u(1,t)\}, \;\;  t\ge 0,
\end{array}\right.
\end{equation}
where $x$ and $t$ denote the independent spatial and time variables, respectively, $u(x, t)$ denotes the cable displacement at $x$ for time $t$,
$m>0$ is the tip mass at the actuator end, $U(t)$ is the boundary control (or input), $F(t)=f(u(\cdot,t),u_t(\cdot,t))+d(t)$,
$f:H^1(0,1)\times L^2(0,1)\rightarrow \mathds{R}$ is the internal uncertainty, $d(t)$ is the external disturbance, $y(t)$ is the observation (output).
System (\ref{beem}) can be used to describe flexible cable with tip mass \cite{Guo2000,He2012,Morgul1994}.

When there is no disturbance ($F(t)\equiv 0$), the velocity feedback control law $U(t)=-\alpha u_t(1,t)$
$(\alpha>0)$ is sufficient for the strong stabilization but not sufficient for the exponential stabilization of the system (\ref{beem}), see \cite{Lee1987}.
By virtue of energy multiplier approach, Morg\"{u}l et al. \cite{Morgul1994} proved that the output feedback control law $U(t)=-\alpha u_t(1,t)-a u_{xt}(1,t)(\alpha,a>0)$ exponentially stabilize the right system. Here the angular velocity $u_{xt}(1,t)$ was adopted,
which is quite different from the case $m=0$ where velocity feedback is sufficient for the exponential stability \cite{Rideau1985}.
Moreover, in the special case $\alpha=m/a\neq 1$, Morg\"{u}l et al. \cite{Morgul1994} verified the Riesz basis generation of the closed-loop system,
and thereby the spectrum-determined growth condition is satisfied.
In \cite{Guo2000}, Guo and Xu showed by virtue of essential spectral analysis that the spectrum-determined growth {\color{blue}condition always holds} for the closed-loop system under output feedback control law $U(t)=-\alpha u_t(1,t)-a u_{xt}(1,t)(\alpha,a>0)$.
However, the Riesz basis generation for the general case is still unsolved. This will be presented in Lemma \ref{Astable}.

Observe that the aforemention literatures used static collocated feedback where the actuators and sensors are located in the same end $x=1$.
The reason why energy multiplier method is successfully used for the stability analysis is that such output feedback makes the system dissipative.
However, it has known for a long time, the performance of closed-loop system under collocated output feedback may be not so good \cite{Cannon1984}.
For the non-collocated setting, sensors can be placed according to the performance requirement,
thereby such design approach for specific systems has been widely used \cite{Cannon1984,Liu2003,Tian2016}.
Since the non-collocated closed-loop system is usually non-dissipative, the well-posedness as well as stability analysis is most difficult.
Dynamic boundary feedback approach may be a better choice to
overcome the difficulty stemmed from non-collocated feedback.
In \cite{Guo2007}, with one measurement $u_{x}(0,t),$
an infinite-dimensional observer-based feedback technique has been employed to construct a stabilizing boundary
feedback controller for a wave equation (\ref{beem}) with $m=0$ and $F(t)=0$.
Such method has been applied to solve non-collocated boundary control designs of Euler-Bernoulli beam equation \cite{Guo2008}.
{\it For (\ref{beem}) with tip mass ($m\neq0$), non-collocated control is still a long standing unsolved problem.}
The reason is that tip mass makes the system much more complex; such system is indeed a hybrid system consists of a PDE and an ODE. One of our task in this paper is to construct a non-collocated output feedback stabilizing controller for system (\ref{beem}) without disturbance. Compared to the existence references \cite{Guo2000,Morgul1994},
{\it our main contributions for the case $F(t)\equiv 0$ are: 1) non-collocated feedback is used, 2) only one measurement is employed,
3) the Riesz basis generations for closed-loop systems in \cite{Guo2000,Morgul1994} have been generalized to the case $m\neq a$,
4) Riesz basis approach is used to prove the exponential stability for a coupled system which is equivalent to the closed-loop system obtained from our feedback control law, where the verification of exponential stabilization as well as Riesz basis generation for coupled equation is most changeable \cite{Guo2007,Guo2008,Guo2019}.}

Uncertainties and disturbances widely exist in various practical engineering control systems.
When the disturbance is taking into consideration ($F(t)\neq 0$), the vibration cable model (\ref{beem}) is quite different from the previous work,
because it is governed by a nonhomogeneous hyperbolic PDE. Even a small disturbance can damage
the stabilizing output feedback design for system without disturbance. Thus, anti-disturbance problem is most changeable: in order
to stabilize the system with disturbance, the controller should be redesigned. Many
engineers and mathematicians focus on developing various approaches to deal with uncertainty and disturbance
in control system problem. In \cite{Morgul1994a} and \cite{Morgul2001}, in order to reject the disturbance under
certain conditions, Morg\"{u}l proposed a kind of dynamic boundary controller for elastic beam and string, respectively.
Guo et al. \cite{GuoW2013a} and Ge et al. \cite{Ge2011b} used adaptive control method to stabilize the wave equation and beam equation with disturbance, respectively. Sliding-mode control (SMC) approach was adopted to reject the disturbance for wave equations \cite{Guo2013a}, beam equation \cite{Guo2013b,Karagiannis2018a}, and the cascade of ODE-wave systems \cite{Liu2017}.
Backstepping approach developed by Krstic and Smyshlyaev \cite{Krstic2006,Krstic2008,Krstic2008b,Krstic2009a,KrsticSmysh2008,Smyshlyaev2005}
is also helpful for reject disturbance of PDEs \cite{Auriol2020,Deutscher2017,Deutscher2019,Deutscher2020,Wang2020} {\color{blue}and output regulation for coupled linear wave-ODE systems (similar to our setup) \cite{Deutscher2020b}}. It should be noted that the active disturbance rejection control (ADRC) proposed by Han \cite{Han2009} is another powerful approach to reject disturbance. The core idea of ADRC is estamaite/cancellation in real time and the key step of control strategy is the construction of an extended state observer (ESO) for the estimating of the state and uncertainty. In the earlier works for PDEs by ADRC like  \cite{Guo2013b,Jin2015,Su2020}, they dealt with the disturbance by ODEs reduced from the associated PDEs through some special test functions.
Therefore, the ESO is of finite-dimension; slow variation, high gains and boundedness of the derivation of the disturbance were used in the ESO.
In \cite{Feng2017a}, Feng and Guo developed a new infinite-dimensional ESO to relax such restricts of the conventional ESO for a class of anti-stable wave equations with external disturbance. Later on, the same method was used to study the stabilization of other wave equations \cite{Mei2020b,Zhou2017a,Zhou2018b} and Euler-Bernoulli beams \cite{Zhou2018a,Zhou2020}.

In section \ref{disturbancestab}, we shall clarify that the estimated state based output feedback control law (\ref{control}) for system (\ref{beem}) with $F(t)\equiv 0$ is not robust to the total disturbance. For system (\ref{beem}) with $f(w,w_t)=0$ and $d(t)$ being bounded, He and Ge \cite{He2012} adopted
adaptive control approach to compensate the uncertainty. In the case $f(w,w_t)=0$ and $m=a\alpha$, by
constructing a finite-dimensional ESO, Xie and Xu \cite{Xie2017} employed active disturbance rejection control approach to cope with the disturbance. However, they excluded the internal uncertainty as well as the general case $m\neq a\alpha$. Moreover, as in \cite{Morgul1994}, the author in \cite{Xie2017} used two measurements including the high order angular velocity feedabck $w_{xt}(1,t)$. In this paper, with two low order measurements $w_{x}(0,t)$ and $w(1,t)$, we shall design an infinite-dimensional ESO to estimate the original state and total disturbance online. By compensating the total disturbance, an
estimated state based output feedback control law is constructed in order to exponentially stabilize the original system. Moreover, all the other states of the closed-loop system are verified to be bounded. {\it The main contribution for the system (\ref{beem}) with disturbance ($F(t)\neq 0$) lies in that, 1) the internal uncertainty is considered, 2) the high order angular velocity feedback $w_{xt}(1,t)$ is avoided, 3) we consider the case $m\neq a,\alpha>0$ while \cite{Xie2017} just solved the special case $\alpha,a>0$ and $m=a\alpha$, 4) our control strategy can be also applied to simplify the existence references \cite{Mei2020b,Zhou2018a,Zhou2018b}.}

Consider system (\ref{beem}) in the energy state Hilbert space $\mathbf{H}_1=H_E^1(0,1)\times L^2(0,1)\times\mathds{C},\ H_E^1(0,1)=\{f|f\in H^1(0,1),f(0)=0\}$, with the following norm
$\|(f,g,$
$\eta)\|^2_{\mathbf{H}_1} = \int_0^1[|f'(x)|^2+|g(x)|^2]dx+\frac{1}{m} |\eta|^2,(f,g,\eta)\in \mathbf{H}_1.$
Define the operator $\mathbf{A}_1:D(\mathbf{A}_1)(\subset \mathbf{H}_1)\rightarrow \mathbf{H}_1$ by
 $\mathbf{A}_1(f,g,\eta)=(g,f'',-f'(1)),$
 $ \forall\;(f,g)\in D(\mathbf{A}_1)=\{(f,g,\eta)\in \mathbf{H}_1| (g,f'',-f'(1))\in \mathbf{H}_1,
   \eta=mg(1)\}.$
Obviously, $\mathbf{A}_1$ is skew-adjoint and generates a unity group.
System (\ref{beem}) is written abstractly by
\begin{align*}
    &\frac{d}{dt}\left(u(\cdot,t),u_t(\cdot,t),mu_t(1,t)\right)\\
&=\mathbf{A}_1\left(u(\cdot,t),u_t(\cdot,t),mu_t(1,t)\right)+\mathbf{B}_1[U(t)+F(t)],
\end{align*}
where $\mathbf{B}_1=(0,0,1)^T$ is a bounded linear operator on $\mathbf{H}_1$. The following
proposition can be derived directly from the proof of \cite[Proposition 1.1]{Zhou2018a}.

\begin{proposition}
Assume that $f:H^1(0,1)\times L^2(0,1)\rightarrow \mathds{R}$ is continuous and satisfies global Lipschitz condition in $H^1(0,1)\times L^2(0,1)$.
Then, for any {\color{blue}$(u(\cdot,0),u_t(\cdot,0),mu_t(1,$
$0))\in \mathbf{H}_1$}, $u,d \in L^2_{loc}(0,\infty)$,
there exists a unique global solution (mild solution) to (\ref{beem}) such that $(u(\cdot,t),u_t(\cdot,t),$
$mu_t(1,t))\in C(0,\infty;\mathbf{H}_1)$.
\end{proposition}

We proceed as follows. In section \ref{section2}, for system (\ref{beem}) without disturbance, we design a Luenberger state observer and an estimated state based stabilizing control law.  We design in section \ref{disturbancestab} an infinite-dimensional ESO for system (\ref{beem}) to estimated total disturbance and state in real time. An estimated total disturbance and estimated state based stabilizing control law is then designed. Moreover, it is proved that an important couple subsystem including the original equation of the closed-loop is exponential stability and the closed-loop system is bounded. In section \ref{shiyan}, some numerical simulations are presented.

\section{Stabilization in the absence of disturbance}\label{section2}

We rewrite (\ref{beem}) with $F(t)\equiv0$ as follows
\begin{equation} \label{beem1}
\left\{\begin{array}{l}
u_{tt}(x,t)=u_{xx}(x,t),\;\; x\in (0,1), \; t>0, \\
u(0,t)=0,u_{x}(1,t)+mu_{tt}(1,t)=U(t), \;  t\ge 0,\\
y(t)=u_x(0,t), \;\;  t\ge 0.
\end{array}\right.
\end{equation}
Our aim in this section is to design a stabilizing control law for system (\ref{beem1}) by virtue of only one noncollocated measurement $u_{x}(0,t)$.
To this end, we first construct an infinite-dimensional Luenberger state observer as follows
\begin{equation} \label{transfer11}
\left\{\begin{array}{l}
\widehat{u}_{tt}(x,t)=\widehat{u}_{xx}(x,t),\;\; x\in (0,1), \; t>0, \\
\widehat{u}_x(0,t)=\gamma\widehat{u}_{t}(0,t)+\beta\widehat{u}(0,t)+u_x(0,t), \\
\widehat{u}_{x}(1,t)+m\widehat{u}_{tt}(1,t)=U(t), \;\;  t\ge 0,
\end{array}\right.
\end{equation}
where $\beta,\gamma>0$.
Set $\widetilde{u}(x,t)=\widehat{u}(x,t)-u(x,t)$ to get
\begin{equation} \label{wwan}
\left\{\begin{array}{l}
\widetilde{u}_{tt}(x,t)=\widetilde{u}_{xx}(x,t),\;\; x\in (0,1), \; t>0, \\
\widetilde{u}_x(0,t)=\gamma\widetilde{u}_{t}(0,t)+\beta\widetilde{u}(0,t), \; \;   t\ge 0,\\
\widetilde{u}_{x}(1,t)+m\widetilde{u}_{tt}(1,t)=0, \;\;  t\ge 0.
\end{array}\right.
\end{equation}
Denote $\mathbf{H}_2=H^1(0,1)\times L^2(0,1)\times \mathds{C}$, with norm $\|(f,g,\eta)\|^2_{\mathbf{H}_2}=\int_0^1|f'(x)|^2+|g(x)|^2
+\beta|f(0)|^2+|\eta|^2/m,$
$(f,g,\eta)\in \mathbf{H}_2.$
Then we can write (\ref{wwan}) abstractly by
\begin{align}\label{wwansemigroup}
    \nonumber&\frac{d}{dt}(\widetilde{u}(\cdot,t),\widetilde{u}_t(\cdot,t),m\widetilde{u}_t(1,t))\\
    &=\mathbf{A}_2(\widetilde{u}(\cdot,t),
    \widetilde{u}_t(\cdot,t),m\widetilde{u}_t(1,t)).
\end{align}
Here $\mathbf{A}_2:D(\mathbf{A}_2)(\subset \mathbf{H}_2)\rightarrow \mathbf{H}_2$ is defined by
$\mathbf{A}_2(f,$
$g,\eta)=(g,f'',-f'(1)),\; \forall\;(f,g,\eta)\in D(\mathbf{A}_2)=\{(f,g,$
$\eta)\in \mathbf{H}_2|(g,f'',-f'(1))\in \mathbf{H}_2,
    f'(0)=\gamma g(0)+\beta f(0),$
    $    \eta=mg(1))\}.$
By \cite{Rao1993}, $\mathbf{A}_2$ generates an exponentially stable $C_0$-semigroup,
{\color{blue}that is, $\|e^{\mathbf{A}_2t}\|\leq M_{\mathbf{A}_2}e^{\omega_{\mathbf{A}_2}t},\; t\geq 0$,
where $M_{\mathbf{A}_2}$ and $\omega_{\mathbf{A}_2}$ are two positive constants.}
The following Lemma \ref{A2stable} presents that the Riesz basis property of $\mathbf{A}_2$ also holds.

\begin{lemma}\label{A2stable}
Assume that $\gamma\neq 1$. Then, there exist a sequence of generalized eigenfunctions of operator $(\mathbf{A}_2,D(\mathbf{A}_2))$ that forms Riesz basis for $\mathbf{H}_2$; the eigenvalues $\{\lambda_n\}_{n=-\infty}^{+\infty}$ have
asymptotic expression
$\lambda_n=\frac{1}{2}{\rm ln}\frac{|\gamma-1|}{\gamma+1}+n_\beta\pi i+O(|n|^{-1}),$
where
$n_\beta$
$=n$ if $\gamma>1$ and $n_\beta=n+1/2$ if $0<\gamma<1$;
the eigenfunction is given by $(f_n,\lambda_nf_n,$
$m\lambda_nf(1))$ with $f_n(x)=[(1+\gamma)\lambda_n+\beta]e^{\lambda_n x}+[(1-\gamma)\lambda_n-\beta]e^{-\lambda_n x}$ and
$F_n(x)$
$=(f_{n}^{'}(x),\lambda_nf_n(x),\beta f_n(0),-\lambda_n^{-1} f'_n(1))/\lambda_n^2=
     \big((1+\gamma)$
     $\left|\frac{\gamma-1}{\gamma+1}\right|^{x/2}e^{in_\beta x}-(1-\gamma)\left|\frac{\gamma+1}{\gamma-1}\right|^{x/2}e^{-in_\beta x},
     (1+\gamma)\left|\frac{\gamma-1}{\gamma+1}\right|^{x/2}$
     $e^{in_\beta x}+(1-\gamma)\left|\frac{\gamma+1}{\gamma-1}\right|^{x/2}e^{-in_\beta x},
   0,0\big)+O(n^{-1}).$
\end{lemma}
\begin{proof}\ \
It follows from \cite{Rao1993} that operator $(\mathbf{A}_2,D(\mathbf{A}_2))$ generates an exponentially stable $C_0$-semigroup,
and it is a densely defined and discrete operator on $\mathbf{H}_2$. Therefore, for any $\lambda\in \sigma(\mathbf{A}_2)=\sigma_P(\mathbf{A}_2)$, ${\rm Re} \lambda<0$. We divide the proof into several steps.

{\bf Step 1:} We claim that for any $\lambda\in\sigma(\mathbf{A}_2)$, there corresponds one eigenfunction
$(f,g,\eta)$ given by
 \begin{align}\label{fxequ300}
 \left\{
   \begin{array}{ll}
      f(x)=[(1+\gamma)\lambda+\beta]e^{\lambda x}+[(1-\gamma)\lambda-\beta]e^{-\lambda x}, \\
     g=\lambda f, \\
     \eta=-\frac{f'(1)}{\lambda}
   \end{array}
 \right.
  \end{align}
and $\lambda$ satisfies the characteristic equation
  \begin{align}\label{tauchar2001}
  e^{2\lambda}[(1+\gamma)\lambda+\beta](1+m\lambda)=[(1-\gamma)\lambda-\beta](1-m\lambda).
  \end{align}
This implies that each eigenvalue of $\mathbf{A}_2$ is geometrically simple.

Indeed, $f''(x)-\lambda^2 f(x)=0$ implies that $f$ is of the form
$f(x)=Ee^{\lambda x}+Fe^{-\lambda x}$.
Use the boundary condition $f'(0)=(\gamma\lambda+\beta)f(0)$ to derive $E=(1+\gamma)\lambda+\beta, F=(1-\gamma)\lambda-\beta$.
We can then derive $\eta=mg(1)=m\lambda f(1)=-\frac{f'(1)}{\lambda}$ and (\ref{tauchar2001}) by virtue of the boundary condition $f'(1)+m\lambda^2 f(1)=0$.

{\bf Step 2:} We show that the eigenvalues $\{\lambda_n\}_{n=-\infty}^{+\infty}$ have
asymptotic expression

\begin{equation}\label{ataunasy100}
    \lambda_n=\frac{1}{2}{\rm ln}\frac{|\gamma-1|}{\gamma+1}+n_\beta\pi i+O(|n|^{-1}),
\end{equation}
where
\begin{align}\label{nbeta}
  n_\beta=\left\{
            \begin{array}{ll}
              n, \gamma>1 \\
              n+\frac{1}{2},0<\gamma<1;
            \end{array}
          \right.
\end{align}
the corresponding eigenfunction $(f_n,\lambda_nf_n,$
$m\lambda_nf(1))$ of $\mathbf{A}_2$ can be chosen such that
 \begin{align}\label{fxequasy300}
     \nonumber&F_n(x)=\\
     \nonumber&\left(
              \begin{array}{c}
                (1+\gamma)\left|\frac{\gamma-1}{\gamma+1}\right|^{x/2}e^{in_\beta x}-(1-\gamma)\left|\frac{\gamma+1}{\gamma-1}\right|^{x/2}e^{-in_\beta x} \\
                (1+\gamma)\left|\frac{\gamma-1}{\gamma+1}\right|^{x/2}e^{in_\beta x}+(1-\gamma)\left|\frac{\gamma+1}{\gamma-1}\right|^{x/2}e^{-in_\beta x} \\
0\\
0\\
              \end{array}
            \right)\\
            &+O(n^{-1}),
  \end{align}
 where
  \begin{equation}\label{fmatrix100}
    F_n(x)=\frac{1}{\lambda_n^2}\begin{pmatrix}
      f_{n}^{'}(x) \\
      \lambda_nf_n(x)\\
      \beta f_n(0)\\
      -\lambda_n^{-1} f'_n(1)\\
    \end{pmatrix}^T.
  \end{equation}

Now we give the proof. By (\ref{tauchar2001}), it follows that
\begin{align}\label{fxc1c2c3c4}
 \nonumber &e^{2\lambda}=\frac{m(\gamma-1)\lambda^2+(m\beta-\gamma+1)\lambda-\beta}{m(\gamma+1)\lambda^2+(m\beta+\gamma+1)\lambda+\beta}\\
&=\frac{\gamma-1}{\gamma+1}+O(|\lambda|^{-1}).
\end{align}
The asymptotic expression (\ref{ataunasy100}) is derived directly by Routhe's theorem.
Then we obtain
\begin{equation}\label{fxequ30}
  \begin{split}
 &  \lambda_n\frac{f_n(x)}{\lambda^2_n}=\frac{(1+\gamma)\lambda_n+\beta}{\lambda_n}e^{\lambda_n x}+\frac{(1-\gamma)\lambda_n-\beta}{\lambda_n}
e^{-\lambda_n x}\\
&=(1+\gamma)\left|\frac{\gamma-1}{\gamma+1}\right|^{x/2}e^{in_\beta x}+(1-\gamma)\left|\frac{\gamma+1}{\gamma-1}\right|^{x/2}e^{-in_\beta x}\\
     &+O(|n|^{-1}),
  \end{split}
 \end{equation}
\begin{equation}\label{fxequ31}
  \begin{split}
   &  \frac{f'_n(x)}{\lambda^2_n}=\frac{(1+\gamma)\lambda_n+\beta}{\lambda_n}e^{\lambda_n x}-\frac{(1-\gamma)\lambda_n-\beta}{\lambda_n}
e^{-\lambda_n x}\\
=&(1+\gamma)\left|\frac{\gamma-1}{\gamma+1}\right|^{x/2}e^{in_\beta x}-(1-\gamma)\left|\frac{\gamma+1}{\gamma-1}\right|^{x/2}e^{-in_\beta x}\\
    & +O(|n|^{-1}),
  \end{split}
 \end{equation}
\begin{equation}\label{fxequ32}
  \begin{split}
   &  \frac{\beta f_n(0)}{\lambda^2_n}=\frac{(1+\gamma)\lambda_n+\beta}{\lambda^2_n}+\frac{(1-\gamma)\lambda_n-\beta}{\lambda^2_n}
=O(|n|^{-1}),
  \end{split}
 \end{equation}
\begin{equation}\label{fxequ33}
  \begin{split}
   &  \frac{-\lambda^{-1}_nf'_n(1)}{\lambda^2_n}=-\frac{(1+\gamma)\lambda_n+\beta}{\lambda^2_n}e^{\lambda_n }+\frac{(1-\gamma)\lambda_n-\beta}{\lambda^2_n}
e^{-\lambda_n }\\
&=O(|n|^{-1}),
  \end{split}
 \end{equation}
The combination of (\ref{fxequ30}), (\ref{fxequ31}), (\ref{fxequ32}) and (\ref{fxequ33}) indicates that the asymptotic expression (\ref{fxequasy300}) of $F_n(x)$ holds.

{\bf Step 3:} We show that there exist a sequence of generalized eigenfunctions of $\mathbf{A}_2$ which forms Riesz basis for $\mathbf{H}_2$.

Now we give the proof. Denote
$$P(x)=\left(
         \begin{array}{cc}
           (1+\gamma)\left|\frac{\gamma-1}{\gamma+1}\right|^{x/2} & -(1-\gamma)\left|\frac{\gamma+1}{\gamma-1}\right|^{x/2} \\
           (1+\gamma)\left|\frac{\gamma-1}{\gamma+1}\right|^{x/2} & (1-\gamma)\left|\frac{\gamma+1}{\gamma-1}\right|^{x/2} \\
         \end{array}
       \right)$$
for $\lambda>1$,
$$P(x)=\left(
         \begin{array}{cc}
           (1+\gamma)\left|\frac{\gamma-1}{\gamma+1}\right|^{x/2}i & (1-\gamma)\left|\frac{\gamma+1}{\gamma-1}\right|^{x/2}i \\
           (1+\gamma)\left|\frac{\gamma-1}{\gamma+1}\right|^{x/2}i & -(1-\gamma)\left|\frac{\gamma+1}{\gamma-1}\right|^{x/2}i \\
         \end{array}
       \right)$$
for $\lambda<1$.
It is easily seen that the operator $\mathbf{T}$ defined by $\mathbf{T}w=P(x)w$ for any $w\in \left(L^2(0,1)\right)^2$, is a bounded and invertible linear operator on $\left(L^2(0,1)\right)^2$. Combine this with the fact that $\{(e^{in\pi x},e^{-in\pi x})\}_{n=-\infty}^{+\infty}$  forms Riesz basis for $\left(L^2(0,1)\right)^2$, to obtain that $\{P(x)(e^{in\pi x},e^{-in\pi x})\}_{n=-\infty}^{+\infty}$ also forms Riesz basis for $\left(L^2(0,1)\right)^2$.
Thereby, $\{(0,0,1,0)\}\bigcup\{(0,0,0,1)\}\bigcup$
$\{(P(x)(e^{in\pi x},e^{-in\pi x}),0,$
$0)$
$\}_{n=-\infty}^{+\infty}$ forms Riesz basis
for $\left(L^2(0,1)\right)^2\times \mathds{C}^2$.

From (\ref{fxequasy300}), we derive
\begin{align}\label{fxequasy3001}
     F_n(x)=(P(x)(e^{in\pi x},e^{-in\pi x}),0,0)+O(n^{-1}),
  \end{align}
which implies that, {\color{blue} there exists} enough big positive integer $N$ such that
\begin{align*}
  &\sum_{|n|>N}\|F_n-(P(x)(e^{in\pi x},e^{-in\pi x}),0,0)\|_{\left(L(0,1)\right)^2\times\mathds{C}^2}^2\\
  &=\sum_{|n|>N}O(|n|^{-2})<\infty.
\end{align*}
Since $\{(0,0,1,0)\}\bigcup\{(0,0,0,1)\}\bigcup\{(P(x)(e^{in\pi x},e^{-in\pi x}),0,0)\}_{n=-\infty}^{+\infty}$ forms Riesz basis
for $\big(L^2(0,$
$1)\big)^2\times \mathds{C}$, it follows by
\cite[Lemma 1]{Guo2001} that there exists a $M\geq N$ such that $\{(0,0,1,0)\}\bigcup\{(0,0,0,1)\}$
$\bigcup\{(P(x)(e^{in\pi x},e^{-in\pi x}),
0,0)\}_{n=-M}^M\bigcup\{F_n\}_{|n|>M}$ forms Riesz basis
for $\left(L^2(0,1)\right)^2\times \mathds{C}^2$.
We define an isometric isomorphism $\mathbb{T}_1:\mathbf{H}_2\rightarrow \left(L^2(0,1)\right)^2\times \mathds{C}^2$
by $\mathbb{T}_1(f,g,\eta)=(f',g,\sqrt{\beta}f(0),\eta),\; \forall\;(f,g,\eta)\in \mathbf{H}_2.$
Then there exist $\{W_n(x),V_n(x),0\}_{n=-M}^M\subset \mathbf{H}_2$ such that $\{(0,0,1)\}\bigcup\{(W_n(x),V_n(x),0)\}_{n=-M}^M$
$\bigcup
\{(\lambda^{-2}_nf_n(x),\lambda^{-1}_nf_n(x),-\lambda_n^{-3}f'_n(1))\}_{|n|>M}$ forms Riesz basis for $\mathbf{H}_2$.
Since $\{(\lambda^{-2}_nf_n(x),\lambda^{-1}_nf_n(x),$
$-\lambda_n^{-3}f'_n(1))\}_{n=-\infty}^\infty$ is a sequence of generalized eigenvectors of $\mathbf{A}_2$,
and by \cite{Rao1993}, $\mathbf{A}_2$ is a densely defined and discrete operator, it follows from \cite[Theorem 1]{Guo2001} that there exists generalized eigenfunctions of $\mathbf{A}_2$ that forms Riesz basis
for $\mathbf{H}_2$. The proof is therefore completed.
\end{proof}

Since the state feedback $U(t)=-\alpha u_t-a u_{xt}(1,t)$ makes the system (\ref{beem1}) exponentially stable \cite{Guo2000,Morgul1994},
Lemma \ref{A2stable} allows us to design the following control law
\begin{align}\label{control}
    U(t)=-\alpha\widehat{u}_t(1,t)-a \widehat{u}_{xt}(1,t).
\end{align}
The closed-loop system is then described by
\begin{align}\label{closednodisturbance}
    \left\{\begin{array}{l}
u_{tt}(x,t)=u_{xx}(x,t),\;\; x\in (0,1), \; t>0, \\
u(0,t)=0,
u_{x}(1,t)+mu_{tt}(1,t)=-\alpha\widehat{u}_t(1,t)\\
-a \widehat{u}_{xt}(1,t), \;\;  t\ge 0,\\
\widehat{u}_{tt}(x,t)=\widehat{u}_{xx}(x,t),\;\; x\in (0,1), \; t>0, \\
{\color{blue}\widehat{u}_x(0,t)=\gamma \widehat{u}_t(0,t)+\beta \widehat{u}(0,t)+u_x(0,t)}, \; \;   t\ge 0,\\
\widehat{u}_{x}(1,t)+m\widehat{u}_{tt}(1,t)=-\alpha\widehat{u}_t(1,t)-a \widehat{u}_{xt}(1,t),
\end{array}\right.
\end{align}
which is equivalent to
\begin{align}\label{closednodisturbance1}
    \left\{\begin{array}{l}
u_{tt}(x,t)=u_{xx}(x,t),\;\; x\in (0,1), \; t>0, \\
u(0,t)=0,
u_{x}(1,t)+mu_{tt}(1,t)=-\alpha u_t(1,t)\\
-a u_{xt}(1,t)
-\alpha\widetilde{u}_t(1,t)-a \widetilde{u}_{xt}(1,t), \;\;  t\ge 0,\\
\widetilde{u}_{tt}(x,t)=\widetilde{u}_{xx}(x,t),\;\; x\in (0,1), \; t>0, \\
\widetilde{u}_x(0,t)=\gamma\widetilde{u}_{t}(0,t)+\beta\widetilde{u}(0,t), \; \;   t\ge 0,\\
\widetilde{u}_{x}(1,t)+m\widetilde{u}_{tt}=0,
\end{array}\right.
\end{align}

Consider system (\ref{closednodisturbance1}) in Hilbert state space $\mathcal{H}=\mathbf{H}$
$\times \mathbf{H}_2$, where $\mathbf{H}=H_E^1(0,1)\times L^2(0,1)\times\mathds{C}$, $H_E^1(0,1)=\{f$
$ \in H^1(0,1):f(0)=0\}$.
The norm of $\mathbf{H}$ is given by
$\|(f,g,\eta)\|^2_{\mathbf{H}} = \int_0^1[|f'(x)|^2+|g(x)|^2]dx+\frac{1}{m+\alpha a} |\eta|^2,$
$(f,g)\in \mathbf{H}.$
Define the operator $\mathbf{A}:D(\mathbf{A})(\subset \mathbf{H})\rightarrow \mathbf{H}$ by
$\mathbf{A}(f,g,\eta)=(g,f'',-f'(1)-\alpha g(1)),\; \forall\;(f,g)\in D(\mathbf{A})=\{(f,g)\in \mathbf{H}| f''(1)=0,
   \eta=mg(1)+af'(1)\}.$
By \cite{Guo2000,Morgul1994}, it follows that $\mathbf{A}$ generates an exponentially stable $C_0$-semigroup,
and there exist a sequence of generalized eigenfunctions of $\mathbf{A}$ that forms Riesz basis for $\mathbf{H}$ provided $\alpha=m/a\neq 1$.
In the following Lemma \ref{Astable}, we generalize the result to the case $a\neq m,\alpha>0$.

\begin{lemma}\label{Astable}
Suppose that $a\neq m$. Then, there exist a sequence of generalized eigenfunctions of operator $(\mathbf{A},D(\mathbf{A}))$ that forms Riesz basis for $\mathbf{H}$; the characteristic equation of $(\mathbf{A},D(\mathbf{A}))$ is $e^{2\lambda}[1+\alpha+(a+m)\lambda]+(1-\alpha)+a-m=0$.
\end{lemma}
\begin{proof}\ \
We divide the proof into several steps.

{\bf Step 1.}
We claim that operator $(\mathbf{A},D(\mathbf{A}))$ a densely defined and discrete operator on $\mathbf{H}$; for any $\lambda\in \sigma(\mathbf{A})=\sigma_P(\mathbf{A})$, ${\rm Re} \lambda<0$.

Indeed, since by \cite{Morgul1994},$(\mathbf{A},D(\mathbf{A}))$ generates an exponentially stable $C_0$-semigroup, it is densely defined and has bounded inverse. Let $(f,g,\eta)\in \mathbf{H}$. Since by Sobolev imbedding theory $H^2(0,1)\times H^1(0,1)\times \mathds{C}$ is compactly imbedding in $H^1(0,1)\times L^2(0,1)\times \mathds{C}$, $\mathbf{A}^{-1}$ is compact. Moreover, for any $\lambda\in \sigma(\mathbf{A})=\sigma_P(\mathbf{A})$, ${\rm Re} \lambda<0$.

{\bf Step 2.}  We claim that for any $\lambda\in\sigma(\mathbf{A})$, there corresponds one eigenfunction
$(f,g,\eta)$ given by
 \begin{align}\label{fxequ300}
 \left\{
   \begin{array}{ll}
      f(x)=e^{\lambda x}-e^{-\lambda x}, \\
     g=\lambda f, \\
     \eta=-\frac{f'(1)+\alpha\lambda f(1)}{\lambda},
   \end{array}
 \right.
  \end{align}
and $\lambda$ satisfies the characteristic equation
  \begin{align}\label{tauchar200}
  e^{2\lambda}[1+\alpha+(a+m)\lambda]+(1-\alpha)+(a-m)=0.
  \end{align}
This implies that each eigenvalue of $\mathbf{A}$ is geometrically simple.

{\bf Step 3:} We show that the eigenvalues $\{\lambda_n\}_{n=-\infty}^{+\infty}$ have
asymptotic expression
\begin{equation}\label{Aataunasy100}
    \lambda_n=\frac{1}{2}{\rm ln}\frac{|m-a|}{m+a}+n_a\pi i+O(|n|^{-1}),
\end{equation}
where $n_a=n$ if $m>a$ and  $n_a=n+\frac{1}{2}$ if $0<m<a$,
the corresponding eigenfunction $(f_n,\lambda_nf_n,$
$m\lambda_nf(1)+af'_n(1))$ of $\mathbf{A}_2$ can be chosen such that
 \begin{align}\label{Afxequasy300}
    \nonumber &F_n(x)=\left(
              \begin{array}{c}
                \left|\frac{m-a}{m+a}\right|^{x/2}e^{in_a x}+\left|\frac{m+a}{m-a}\right|^{x/2}e^{-in_a x} \\
                \left|\frac{m-a}{m+a}\right|^{x/2}e^{in_a x}-\left|\frac{m+a}{m-a}\right|^{x/2}e^{-in_a x} \\
0\\
              \end{array}
            \right)\\
            &+O(n^{-1}),
  \end{align}
 where
  \begin{equation}\label{Afmatrix100}
    F_n(x)=\frac{1}{\lambda_n}\begin{pmatrix}
      f_{n}^{'}(x) \\
      \lambda_nf_n(x)\\
      -\lambda_n^{-1}[f'_n(1)+\alpha\lambda_n f_n(1)]\\
    \end{pmatrix}^T.
  \end{equation}

Now we give the proof. By (\ref{tauchar200}), it follows that
\begin{align}\label{fxc1c2c3c4}
e^{2\lambda}=-\frac{1-\alpha+(a-m)\lambda}{1+\alpha+(a+m)\lambda}=\frac{m-a}{m+a}+O(|\lambda|^{-1}).
\end{align}
The asymptotic expression (\ref{Aataunasy100}) is derived directly by Routhe's theorem.

{\bf {\color{blue}Step 4}:} We show that there exist a sequence of generalized eigenfunctions of $\mathbf{A}$ which forms Riesz basis for $\mathbf{H}$.

Now we give the proof. Denote
\begin{align*}
  Q(x)=\left\{
         \begin{array}{ll}
           \left(
             \begin{array}{cc}
               \left|\frac{m-a}{m+a}\right|^{x/2} &\left|\frac{m+a}{m-a}\right|^{x/2}\\
               \left|\frac{m-a}{m+a}\right|^{x/2} & -\left|\frac{m+a}{m-a}\right|^{x/2} \\
             \end{array}
           \right), m>a,
 \\
           \left(
             \begin{array}{cc}
               \left|\frac{m-a}{m+a}\right|^{x/2}i &-\left|\frac{m+a}{m-a}\right|^{x/2}i \\
               \left|\frac{m-a}{m+a}\right|^{x/2}i & \left|\frac{m+a}{m-a}\right|^{x/2}i\\
             \end{array}
           \right),m<a.
         \end{array}
       \right.
\end{align*}
It is easily seen that the operator $\mathbf{T}$ defined by $\mathbf{T}w=Q(x)w$ for any $w\in \left(L^2(0,1)\right)^2$, is a bounded and invertible operator
on $\left(L^2(0,1)\right)^2$. Combine this with the fact that $\{(e^{in\pi x},e^{-in\pi x})\}_{n=-\infty}^{+\infty}$  forms Riesz basis for $\left(L^2(0,1)\right)^2$, to obtain that $\{Q(x)(e^{in\pi x},e^{-in\pi x})\}_{n=-\infty}^{+\infty}$ also forms Riesz basis for $\left(L^2(0,1)\right)^2$.
Thereby, $\{(0,0,1)\}\bigcup\{(Q(x)(e^{in\pi x},e^{-in\pi x}),$
$0)\}_{n=-\infty}^{+\infty}$ forms Riesz basis
for $\left(L^2(0,1)\right)^2\times \mathds{C}$.

By (\ref{fxequasy300}), it follows that
\begin{align}\label{fxequasy3001}
     F_n(x)=(Q(x)(e^{in\pi x},e^{-in\pi x}),0)+O(n^{-1}),
  \end{align}
which implies that, there exists enough big positive integer $N$ such that
\begin{align*}
  \sum_{|n|>N}\|F_n(x)-(Q(x)(e^{in\pi x},e^{-in\pi x}),0)\|_{\left(L(0,1)\right)^2\times\mathds{C}^2}^2=\sum_{|n|>N}O(|n|^{-2})<\infty.
\end{align*}
{\color{blue}Since $\{(0,0,1)\}\bigcup\{(Q(x)(e^{in\pi x},e^{-in\pi x}),0)\}_{n=-\infty}^{+\infty}$ forms Riesz basis
for $\left(L^2(0,1)\right)^2\times \mathds{C}$, it follows by
\cite[Lemma 1]{Guo2001} that there exists a $M\geq N$ such that $\{(0,0,1)\}\bigcup\{(Q(x)(e^{in\pi x},e^{-in\pi x}),
0)\}_{n=-M}^M\bigcup$
 $\{F_n\}_{|n|>M}$ forms Riesz basis
for $\left(L^2(0,1)\right)^2\times \mathds{C}$.
We define an isometric isomorphism $\mathbb{T}_1:\mathbf{H}\rightarrow \left(L^2(0,1)\right)^2\times \mathds{C}$
by $\mathbb{T}_1(f,g,\eta)=(f',g,\eta),\; \forall\;(f,g,\eta)\in \mathbf{H}.$
Then there exist $\{W_n(x),V_n(x),0\}_{n=-M}^M\subset \mathbf{H}$ such that $\{(0,0,1)\}\bigcup\{(W_n(x),V_n(x),0)\}_{n=-M}^M\bigcup
\{(\lambda^{-1}_nf_n(x),f_n(x),$
$-\lambda_n^{-2}[f'_n(1)+\alpha\lambda_n f_n(1)])$
$\}_{|n|>M}$ forms Riesz basis for $\mathbf{H}$.
Since $\{(\lambda^{-1}_n(x),f_n(x),$
$-\lambda_n^{-2}[f'_n(1)+\alpha\lambda_n f_n(1)])\}_{n=-\infty}^\infty$ is a sequence of generalized eigenvectors of $\mathbf{A}$,
and $\mathbf{A}$ is a densely defined and discrete operator.}
By \cite[Theorem 1]{Guo2001}, it follows that there exists generalized eigenfunctions of $\mathbf{A}$ that forms Riesz basis
for $\mathbf{H}$. This completes the proof.
\end{proof}

Define operator $\mathcal{A}:D(\mathcal{A})(\subset \mathcal{H})\rightarrow \mathcal{H}$ by
$\mathcal{A}(f,g,\eta,\phi,\psi,$
$h)=(g,f'',-f'(1)-\alpha g(1)-\alpha \psi(1),
    \mathbf{A}_2(\phi,\psi,h)),$
    $D(\mathcal{A})=\{(f,g,\eta,\phi,\psi,h)\in (H_E^1(0,1)\bigcap $
    $H^2(0,1))
    \times H_E^1(0,1)\times \mathds{C}\times D(\mathbf{A}_2),
    \eta=af'(1)+mg(1)+a\phi'(1),h=m\psi(1) \}.$
Then the system (\ref{closednodisturbance1}) is abstractly described by
$$\frac{d}{dt}Z(t)=\mathcal{A}Z(t),$$
where $Z(t)=\big(u(\cdot,t),u_t(\cdot,t),au_x(1,t)+mu_t(1,t)+a\widetilde{u}_x(1,t),\widetilde{u}(\cdot,t),$
$\widetilde{u}_t(\cdot,t),m\widetilde{u}_t(1,t)\big).$
We shall use Riesz basis approach to verify semigroup generation and exponential stability.
However, it is difficult to verify the Riesz basis generation of couple wave equations \cite{Guo2007,Guo2019,Mei2020b}, not to mention wave equations
with tip mass. We shall adopt Bari's theorem to prove Riesz basis generation of operator $\mathcal{A}$ by finding out complicated
relations between sequences of generalized eigenfunctions.

\begin{theorem}\label{exponentialnodisturbance}
Assume that $m\neq a$, $m\neq a\gamma$ and $\gamma\neq 1$. Then, the operator $\mathcal{A}$ generates an exponentially stable $C_0$-semigroup on $\mathcal{H}$.
\end{theorem}
\begin{proof}\ \
It is easy to show that $\mathcal{A}^{-1}$ exists and is compact on $\mathcal{H}$, thereby the spectrum of $\mathcal{A}$ consists of eigenvalues.
Now we show that $\sigma(\mathcal{A})=\sigma(\mathbf{A})\bigcup \sigma(\mathbf{A}_2)$.
By \cite{Guo2007} and Lemma \ref{Riesz}, $\mathbf{A}$ and $\mathbf{A}_2$ are also discrete operators. Hence $\sigma(\mathcal{A})\supseteq\sigma(\mathbf{A})\bigcup \sigma(\mathbf{A}_2).$
Let $\lambda\in \sigma(\mathcal{A})$ and $(f,g,\eta,\phi,\psi,h)$ be the corresponding eigenfunction.
If $(\phi,\psi,h)\neq 0$, $\lambda\in\sigma(\mathbf{A}_2)$; if $(\phi,\psi,h)=0$, we have that $(f,g,\eta)\neq 0$,
{\color{blue}$(f,g,\eta)\in D(\mathbf{A})$}, and {\color{blue}$\lambda (f,g,\eta)=\mathbf{A}(f,g,\eta)$}, which implies that {\color{blue}$\lambda\in \sigma(\mathbf{A})$}. Hence $\sigma(\mathcal{A})\subseteq\sigma(\mathbf{A})\bigcup \sigma(\mathbf{A}_2).$ Therefore $\sigma(\mathcal{A})=\sigma(\mathbf{A})\bigcup \sigma(\mathbf{A}_2)$.

Next, we shall show that the generalized eigenfunction of $\mathcal{A}$ forms a Riesz basis for $\mathcal{H}$.
Let $\{\mu_{n}\}_{n=-\infty}^\infty$ and $\{\lambda_n\}_{n=-\infty}^\infty$ be respectively the eigenvalues of $\mathbf{A}$ and $\mathbf{A}_2$.
Let $\{(\mu_{n}^{-1}f_{n},f_{n},-\mu_n^{-2}( f'(1)$
$+\alpha\mu_nf_n(1)))\}_{n=-\infty}^\infty$ and $\{(\lambda^{-2}_n\phi_{n},\lambda_n^{-1}\phi_{n},\lambda^{-3}_n\phi'_{n}(1)\}_{n=-\infty}^\infty$ be the generalized eigenfunctions corresponding to $\{\mu_{n}\}_{n=-\infty}^\infty$ and $\{\lambda_{n}\}_{n=-\infty}^\infty$ such that they
form Riesz basises for $\mathbf{H}$ and $ \mathbf{H}_2$, respectively.
Hence, the sequence $\{(\mu_{n}^{-1}f_{n},f_{n},-\mu_n^{-2}( f'(1)+\alpha\mu_nf_n(1)),0,0,0)\}_{n=-\infty}^\infty\bigcup$
$\{(0,0,0,\lambda^{-2}_n\phi_{n},\lambda_n^{-1}\phi_{n},$
$\lambda^{-3}_n\phi'_{n}(1))\}_{n=-\infty}^\infty$ forms a Riesz basis for $\mathcal{H}$,
 which is equivalent to that $\{(\mu_{n}^{-1}f'_{n},f_{n},-\mu_n^{-2}$
 $( f'(1)+\alpha\mu_nf_n(1)),0,0,0,0)\}_{n=-\infty}^\infty\bigcup\{(0,0,0,\lambda^{-2}_n\phi'_{n},$
 $\lambda_n^{-1}\phi_{n},\lambda_n^{-2}\beta\phi_n(0),
\lambda^{-3}_n\phi'_{n}(1)$
$\}_{n=-\infty}^\infty$
forms a Riesz basis for $\left(L^2(0,1)\right)^2\times \mathds{C}\times \left(L^2(0,1)\right)^2\times \mathds{C}^2$.

Let $\lambda\in \sigma(\mathcal{A})$ and $\big(\lambda^{-2}f,\lambda^{-1}f,
-\lambda^{-3}(f'(1)+\alpha\lambda f(1)+\alpha\lambda \phi(1)),$
$\lambda^{-2}\phi,\lambda^{-1}\phi,\lambda^{-3}\phi'(1)\big)$ be the corresponding eigenfunction. If $\phi=0$, then $\big(\lambda^{-2}f,\lambda^{-1}f, -\lambda^{-3}(f'(1)+\alpha\lambda f(1))\big)\neq 0$ and $\lambda\in \sigma(\mathbf{A})$.
Hence the eigenvalues $\{\mu_{n}\}_{n=1}^\infty$ corresponds the eigenfunctions $\{\big(\mu_{n}^{-2}f_n,\mu_{n}^{-1}f_n,$
$-\mu_{n}^{-3}(f_n'(1)+\alpha\mu_{n} f_n(1)),0,0,0\big)\}_{n=-\infty}^\infty$.

If $\phi\neq 0$, then $\lambda\in \sigma(\mathbf{A}_2)$. The eigenvalues $\{\lambda_{n}\}_{n=-\infty}^\infty$ corresponds
the eigenfunction $\big(\lambda_n^{-2}\phi_n,$
$\lambda_n^{-1}\phi_n,\lambda_n^{-3}\phi_n'(1)\big)$ of $\mathbf{A}_2$. By Lemma \ref{A2stable}, it follows that
\begin{equation}\label{f1n}
  \begin{split}
     \phi_{n}(x)=[(1+\gamma)\lambda_n+\beta]e^{\lambda_n x}+[(1-\gamma)\lambda_n-\beta]e^{-\lambda_n x}.
  \end{split}
  \end{equation}
Let $\big(\lambda_n^{-2}f_{1n},\lambda_n^{-1}f_{1n},
-\lambda_n^{-3}(f_{1n}'(1)+\alpha\lambda_n f_{1n}(1)+\alpha\lambda_n \phi_{n}(1)),\lambda_n^{-2}\phi_n,$
$\lambda_n^{-1}\phi_n,\lambda_n^{-3}\phi_n'(1)\big)\in \mathcal{H}$ be the eigenfunction of $\mathcal{A}$ corresponding to eigenvalues $\{\lambda_{n}\}_{n=-\infty}^\infty$. Then, $f_{1n}$ satisfies
$f''_{1n}(x)=\lambda_n^2f_{1n}(x),$
        $f_{1n}(0)=0,
       (1+a\lambda_n)f'_{1n}(1)+(m\lambda^2_n+\alpha\lambda_n)f_{1n}
       =-\alpha\lambda_n\phi_n(1)-a\lambda_n\phi'_n(1).$
The solution of the ode $f''_{1n}(x)=\lambda^2_{n}f_{1n}(x)$ with boundary conditions $f_{1n}(0)=0$ is of the form
\begin{align}\label{f1nE}
    f_{1n}(x)=E(e^{\lambda_n x}-e^{-\lambda_n x}).
\end{align}
Combine (\ref{f1n}) and $(1+a\lambda_n)f'_{1n}(1)+(m\lambda^2_n+\alpha\lambda_n)f_{1n}=-\alpha\lambda_n\phi_n(1)-a\lambda_n\phi'_n(1)$ to derive
$\frac{E}{\lambda_n}=-\frac{\alpha\phi_n(1)+a\phi'_n(1)}{E_{11}e^{\lambda_n}+E_{12}e^{-\lambda_n}}$
$=-\frac{E_{21}e^{\lambda_n}
   +E_{22}e^{-\lambda_n}}{E_{11}e^{\lambda_n}+E_{12}e^{-\lambda_n}},$
where $E_{11}=\lambda_n[1+\alpha+(a+m)\lambda_n],E_{12}=\lambda_n[1-\alpha+(a-m)\lambda_n], E_{21}=(\alpha$
$+a\lambda_n)[(1+\gamma)\lambda_n+\beta],E_{22}=(\alpha-a\lambda_n)[(1-\gamma)\lambda_n-\beta].$
Use the characteristic equation in Lemma \ref{Astable} to get
$\frac{E}{\lambda_n}=-\frac{E_{21}F_1+E_{22}F_2}
   {E_{11}F_1+E_{12}F_2}=\frac{a(\gamma^2-1)}{\gamma a-m}+O(|n^{-1}|),$
where $F_1=[(1-\gamma)\lambda_n-\beta](1-m\lambda_n)$, $F_2=[(1+\gamma)\lambda_n+\beta](1+m\lambda_n)$.
This, together with (\ref{f1nE}) indicates
\begin{align*}
&\lambda_{n}^{-2}f'_{1n}(x)=\lambda_n^{-1}E[e^{\lambda_nx}+e^{-\lambda_nx}]=\frac{a(\gamma^2-1)}{\gamma a-m}\cdot\\
&\bigg[\left|\frac{\gamma-1}{\gamma+1}\right|^{x/2}e^{in_\beta x}+\left|\frac{\gamma+1}{\gamma-1}\right|^{x/2}e^{-in_\beta x}\bigg]+O(|n|^{-1}),\\
&\lambda_{n}^{-1}f_{1n}(x)=\lambda_n^{-1}E[e^{\lambda_nx}-e^{-\lambda_nx}]=\frac{a(\gamma^2-1)}{\gamma a-m}\cdot\\
&\left[\left|\frac{\gamma-1}{\gamma+1}\right|^{x/2}e^{in_\beta x}-\left|\frac{\gamma+1}{\gamma-1}\right|^{x/2}e^{-in_\beta x}\right]+O(|n|^{-1}).
\end{align*}
Denote $W=\left(
            \begin{array}{cc}
              I_3 & V \\
              0 & I_3 \\
            \end{array}
          \right),
$ with $V=\frac{2a}{m-a\gamma}\left(
            \begin{array}{ccc}
              -\gamma & 1& 0 \\
              1 & -\gamma & 0 \\
              0& 0& 0\\
            \end{array}
          \right).$
Obviously, $W$ has bounded inverse. Combine this with Lemma \ref{A2stable} to derive
\begin{align}
\nonumber    &\big(\mu_{n}^{-2}f_n,\mu_{n}^{-1}f_n,
-\mu_{n}^{-3}(f_n'(1)+\alpha\mu_{n} f_n(1)),0,0,0\big)^T=\\
\label{guanxi1}&W\big(\mu_{n}^{-2}f_n,\mu_{n}^{-1}f_n,
-\mu_{n}^{-3}(f_n'(1)+\alpha\mu_{n} f_n(1)),0,0,0\big)^T,\\
\nonumber&\big(\lambda_n^{-2}f_{1n},\lambda_n^{-1}f_{1n},
-\lambda_n^{-3}(f_{1n}'(1)+\alpha\lambda_n f_{1n}(1)\\
\nonumber&+\alpha\lambda_n \phi_{n}(1)),\lambda_n^{-2}\phi_n,\lambda_n^{-1}\phi_n,\lambda_n^{-3}\phi_n'(1)\big)^T=W\big(0,0,0,\\
\label{guanxi2}&\lambda_n^{-2}\phi_n,\lambda_n^{-1}\phi_n,\lambda_n^{-3}\phi_n'(1)\big)^T+O(|n|^{-1}).
\end{align}
Then, by Bari's theorem and \cite[Theorem 1]{Guo2001}, the sequence $\{\big(\mu_{n}^{-2}f'_n,\mu_{n}^{-1}f_n,$
$-\mu_{n}^{-3}(f_n'(1)+\alpha\mu_{n} f_n(1)),$
$0,0,0,0\big)\}_{n=-\infty}^\infty \bigcup \{\big(\lambda_n^{-2}f'_{1n},$
$\lambda_n^{-1}f_{1n},-\lambda_n^{-3}(f_{1n}'(1)+\alpha\lambda_n f_{1n}(1)+\alpha\lambda_n \phi_{n}(1)),\lambda_n^{-2}\phi'_n,$
$\lambda_n^{-1}\phi_n,\lambda_n^{-2}\beta \phi_n(0),$
$\lambda_n^{-3}\phi_n'(1)\big)\}_{n=-\infty}^\infty$ forms Riesz basis for the Hilbert space $\big(L^2(0,1)\big)^2\times  \mathds{C}\times \left(L^2(0,1)\right)^2
\times  \mathds{C}^2$,
which is equivalent to that the sequence $\{\big(\mu_{n}^{-2}f_n,\mu_{n}^{-1}f_n,$
$-\mu_{n}^{-3}(f_n'(1)+\alpha\mu_{n} f_n(1)),0,0,0\big)\}_{n=-\infty}^\infty $
$\bigcup \{\big(\lambda_n^{-2}f_{1n},$
$\lambda_n^{-1}f_{1n},
-\lambda_n^{-3}(f_{1n}'(1)+\alpha\lambda_n f_{1n}(1)+\alpha\lambda_n \phi_{n}(1)),\lambda_n^{-2}\phi_n,$
$\lambda_n^{-1}\phi_n,\lambda_n^{-3}\phi_n'(1)\big)\}_{n=-\infty}^\infty$ forms Riesz basis for $\mathcal{H}$.

The semigroup generation and spectrum-determined growth condition of $\mathcal{A}$ are derived directly from the Riesz basis property.
Since both $\mathbf{A}$ and $\mathbf{A}_2$ are generator of exponentially stable $C_0$-semigroups,  $e^{\mathcal{A}t}$ is exponentially stable.
\end{proof}

\begin{remark}\label{zhishu}\em
We have to mention that, although it forms Riesz basis
for $\mathcal{H}$,
$\{(\mu_{n}^{-1}f_{n},f_{n},-\mu_n^{-2}( f'(1)+\alpha\mu_nf_n(1)),$
$0,0,0)\}_{n=-\infty}^\infty\bigcup\{(0,0,0,\lambda^{-2}_n\phi_{n},\lambda_n^{-1}\phi_{n},
\lambda^{-3}_n\phi'_{n}(1)\}_{n=-\infty}^\infty$  are not all the generalized eigenfunctions of $\mathcal{A}$.
This is the reason why we choose to compute the generalized eigenfunctions of $\mathcal{A}$
and find out the relations (\ref{guanxi1}) and (\ref{guanxi2}), which are the key steps for the proof.
\end{remark}

\begin{theorem}\label{nodisturbancemain}
Given initial value $(u(\cdot,0),u_t(\cdot,0),mu_t(1,$
$0)+a\widehat{u}_x(1,0),\widehat{u}(\cdot,0),\widehat{u}_t(\cdot,0),m\widehat{u}_t(1,0)+a\widehat{u}_x(1,0)) \in \mathcal{H}_1$, there exists a unique solution to system (\ref{closednodisturbance}) such that
$(u(\cdot,t),u_t(\cdot,t),mu_t(1,t)+a\widehat{u}_x(1,t),\widehat{u}(\cdot,t),\widehat{u}_t(\cdot,t),m\widehat{u}_t(1,$
$t)+a\widehat{u}_x(1,t)) \in C(0,\infty;
\mathcal{H}_1)$  satisfying
$\int_0^1(|u_t(x,t)|^2+|u_{x}(x,t)|^2+|\widehat{u}_t(x,t)|^2+|\widehat{u}_{x}(x,t)|^2)dx
  +|mu_t(1,t)+a\widehat{u}_x(1,t)|^2/m+{\color{blue}|m\widehat{u}_t(1,0)+a\widehat{u}_x(1,0)|^2/(m+a\alpha)}+\beta|\widehat{u}(0,t)|^2$
  $\leq M_1e^{-\gamma_1 t},$
where $M_1$ and $\gamma_1$ are two positive constants.
\end{theorem}
\begin{proof}\ \
Fix $(u(\cdot,0),u_t(\cdot,0),mu_t(1,0)+a\widehat{u}_x(1,0),$
$\widehat{u}(\cdot,0),\widehat{u}_t(\cdot,0),m\widehat{u}_t(1,0)+a\widehat{u}_x(1,0)) \in \mathbf{H}_1\times \mathbf{H}_3$.
Then we have $(u(\cdot,0),u_t(\cdot,0),mu_t(1,0)+a[u_x(1,0)+\widetilde{u}_x(1,0)],\widetilde{u}(\cdot,0),\widetilde{u}_t(\cdot,0),m\widetilde{u}_x(1,0)) \in \mathbf{H}\times \mathbf{H}_2$. By Theorem \ref{exponentialnodisturbance},
it follows that $(u(\cdot,t),u_t(\cdot,t),mu_t(1,t)+a[u_x(1,t)+\widetilde{u}_x(1,t)],\widetilde{u}(\cdot,t),\widetilde{u}_t(\cdot,t),m\widetilde{u}_x(1,t))
=e^{\mathcal{A}t}(u(\cdot,$
$0),u_t(\cdot,0),mu_t(1,0)+a[u_x(1,0)+\widetilde{u}_x(1,0)],\widetilde{u}(\cdot,0),$
$\widetilde{u}_t(\cdot,0),m\widetilde{u}_x(1,0))
\in C(0,\infty;\mathbf{H}\times \mathbf{H}_2)$ and
\begin{align*}
  &\int_0^1(|u_t(x,t)|^2+|u_{x}(x,t)|^2+|\widehat{u}_t(x,t)|^2+|\widehat{u}_{x}(x,t)|^2)dx\\
  &+\frac{1}{m}|mu_t(1,0)+a\widehat{u}_x(1,t)|^2+\frac{1}{m+a\alpha}|m\widehat{u}_t(1,0)\\
  &+a\widehat{u}_x(1,t)|^2+\beta|\widehat{u}(0,t)|^2\\
  &\leq \int_0^1(|u_t(x,t)|^2+|u_{x}(x,t)|^2+2|u_t(x,t)|^2\\
  &+2|u_{x}(x,t)|^2+2|\widetilde{u}_t(x,t)|^2+2|\widetilde{u}_{x}(x,t)|^2)dx\\
  &+\frac{1}{m}|mu_t(1,t)+a[u_x(1,t)+\widetilde{u}_x(1,t)]|^2\\
  &+\frac{2}{m+a\alpha}|mu_t(1,t)+a[u_x(1,t)+\widetilde{u}_x(1,t)]|^2\\
  &+\frac{2}{m+a\alpha}|m\widetilde{u}_t(1,t)|^2+\beta|\widetilde{u}(0,t)|^2\\
 &\leq 3\left[1+\frac{m+a\alpha}{m}\right]\|(u(\cdot,t),u_t(\cdot,t),mu_t(1,t)\\
  &+a[u_x(1,t)+\widetilde{u}_x(1,t)],\widetilde{u}(\cdot,t),\widetilde{u}_t(\cdot,t),m\widetilde{u}_x(1,t))\|^2\\
&=3\left[1+\frac{m+a\alpha}{m}\right]\|e^{\mathcal{A}t}(u(\cdot,0),u_t(\cdot,0),mu_t(1,0)\\
&+a[u_x(1,0)+\widetilde{u}_x(1,0)],\widetilde{u}(\cdot,0)\|^2\\
    &\leq M_1e^{-\gamma_1 t},
\end{align*}
where $M_1=3M^2_{\mathcal{A}}\left[1+\frac{m+a\alpha}{m}\right]\|(u(\cdot,0),u_t(\cdot,0),mu_t(1,$
$0)+a[u_x(1,0)+\widetilde{u}_x(1,0)],\widetilde{u}(\cdot,0)\|^2$, $\gamma_1=2{\color{blue}\omega_\mathcal{A}}$.
This completes the proof.
\end{proof}

\begin{remark}\em
Theorem \ref{nodisturbancemain} presents not only the exponential stability of $(u(\cdot),u_t(\cdot),\widehat{u}(\cdot,t),\widehat{u}_t(\cdot,t))$
but also the exponential stability of $\eta(t)=mu_t(1,t)+a[u_x(1,t)+\widetilde{u}_x(1,t)]$ and $\psi(t)=\widehat{u}_t(1,t)+a[u_x(1,t)+\widetilde{u}_x(1,t)]$,
because $\eta(t)$ and $\psi(t)$ are the states of the boundary ode dynamic parts of the closed-loop system.
{\color{blue}This is quite different} from the case in \cite{Guo2007} where $m=0$.
By Theorem \ref{nodisturbancemain}, with only one {\color{blue}non-collocated measurement $u_{x}(0,t)$} we can design
the estimated state based controller (\ref{control}) to exponentially stabilize the vibration cable with tip mass (\ref{beem1}).
Therefore, we improve the results in references \cite{Guo2000,Morgul1994}, where two collocated measurements $u_t(1,t)$ and $u_{xt}(1,t)$ with $u_{xt}(1,t)$ being angular velocity were used.
\end{remark}

\begin{remark}\em
In \cite{Guo2007}, Guo and Xu verified the exponential stability of closed-loop system (\ref{closednodisturbance}) with $m=0$.
By virtue of Riesz basis approach, the authors used 2.5 pages of two columns to verify the semigroup generation and exponential stability of system \cite[(3.1)]{Guo2007}. When $m\neq 0$, we transfer the closed-loop system to an equivalent system (\ref{closednodisturbance1}) which contains the original system and error systems. Although we also used Riesz basis approach and our case is more complicated ($m\neq 0$), our proof is much shorter than that in \cite{Guo2007} because (\ref{closednodisturbance1}) is easier to be dealt with:
it contains an independent subsystem related to $\widetilde{u}$. Our method {\color{blue}can be used to} simplify the verification of exponential stability in \cite{Guo2007}.
\end{remark}

\section{Stabilization in presence of internal uncertainty and external disturbance}\label{disturbancestab}

When we consider disturbance $F(t)\equiv F$, the boundary condition $u_{x}(1,t)+mu_{tt}(1,t)=-\alpha\widehat{u}_t(1,t)-a \widehat{u}_{xt}(1,t)$ of {\color{blue}the closed-loop system (\ref{closednodisturbance}) is changed to} $u_{x}(1,t)+mu_{tt}(1,t)=-\alpha\widehat{u}_t(1,t)-a \widehat{u}_{xt}(1,t)
+F$. One can see that $u(x,t)=Fx,\widehat{u}(x,t)=-F/\beta$ is a solution of the closed-loop system, which is not stable.
Therefore, when there is disturbance, the stabilizing control law should be redesigned.
Motivated by \cite{Feng2017a}, for system (\ref{beem1}) with $F(t)\neq 0$, we propose in this section an infinite-dimensional ESO without high gain described as follow
\begin{equation} \label{transfer}
\left\{\begin{array}{l}
v_{tt}(x,t)=v_{xx}(x,t),\;\; x\in (0,1), \; t>0, \\
v_x(0,t)=\gamma v_{t}(0,t)+\beta v(0,t)+u_x(0,t), \\
v_{x}(1,t)+mv_{tt}(1,t)=U(t), \;\;  t\ge 0, \\
q_{tt}(x,t)=q_{xx}(x,t),\;\; x\in (0,1), \; t>0, \\
q_x(0,t)=\gamma q_{t}(0,t)+\beta q(0,t), \; \;  t\ge 0,\\
q(1,t)=v(1,t)-u(1,t), \;\;  t\ge 0.
\end{array}\right.
\end{equation}
{\color{blue} Since it depends only on the input $u(t)$ and output $u_x(0,t),u(1,t)$, system (\ref{transfer}) is completely known.}
Although there exists disturbance, the high order angular velocity $u_{xt}(1,t)$ is not used.

Set $\widehat{v}(x,t)=v(x,t)-u(x,t)$ to derive the error equation with unknown input as follows
\begin{equation} \label{perror}
\left\{\begin{array}{l}
\widehat{v}_{tt}(x,t)=\widehat{v}_{xx}(x,t),\;\; x\in (0,1), \; t>0, \\
\widehat{v}_x(0,t)=\gamma \widehat{v}_t(0,t)+\beta \widehat{v}(0,t), \; \;   t\ge 0,\\
\widehat{v}_{x}(1,t)+m\widehat{v}_{tt}(1,t)=-F(t), \;\;  t\ge 0
\end{array}\right.
\end{equation}
which is written abstractly by
\begin{align}\label{perrorsemigroup}
    \frac{d}{dt}{\color{blue}(\widehat{v}(\cdot,t),\widehat{v}_t(\cdot,t),m\widehat{v}_t(1,t))}=\mathbf{A}_2{\color{blue}(\widehat{v}(\cdot,t),\widehat{v}_t(\cdot,t),m\widehat{v}_t(1,t))}-\mathbf{B}_{2}F(t),
\end{align}
where $\mathbf{B}_2=(0,0,1)^T$ is a bounded linear operator.
Since $\mathbf{A}_2$ is a generator of an exponentially stable $C_0$-semigroup, the role of the $v$-part of the system (\ref{transfer}) is to transfer the total disturbance $F(t)$ into an exponentially stable system that is relatively easier to be dealt with.

Similar to the proof of \cite[Lemma A.1 and A.2]{Zhou2018a}, we derive the following lemma.

\begin{lemma}\label{admissible}
Assume that $d\in L^\infty(0,\infty)$ (or $d\in L^2(0,\infty))$, $f: H^1(0,1)\times L^2(0,1)\rightarrow \mathds{R}$ is continuous and system (\ref{beem}) admits a unique bounded solution ${\color{blue}(u(\cdot,t),u_t(\cdot,t))}\in C(0,\infty;H^1(0,1)\times L^2(0,1))$. Then for any initial value ${\color{blue}(\widehat{v}(\cdot,0),\widehat{v}_t(\cdot,0),m\widehat{v}_t(1,0))}\in \mathbf{H}_2$, there exists a unique solution
to system (\ref{perror}) such that $(\widehat{v}(\cdot,t),\widehat{v}_t(\cdot,t),m\widehat{v}_t(1,t))\in C(0,\infty;\mathbf{H}_2)$ and
$\|(\widehat{v}(\cdot,t),\widehat{v}_t(\cdot,t),m\widehat{v}_t(1,t))\|_{\mathbf{H}_2}< +\infty$.
Moreover, if $\lim_{t\rightarrow \infty}f(w(\cdot,t),w_t(\cdot,t))=0$ and $d\in L^2(0,\infty)$,
then $\lim_{t\rightarrow \infty}\|(\widehat{v}(\cdot,t),\widehat{v}_t(\cdot,t),m\widehat{v}_t(1,$
$t))\|_{\mathbf{H}_2}=0$.
\end{lemma}

Set $\widehat{q}(x,t)=q(x,t)-\widehat{v}(x,t)$ to get
\begin{equation} \label{erroerrorobserver}
\left\{\begin{array}{l}
\widehat{q}_{tt}(x,t)=\widehat{q}_{xx}(x,t),\;\; x\in (0,1), \; t>0, \\
{\color{blue}\widehat{q}_x(0,t)=\gamma\widehat{q}_t(0,t)+\beta\widehat{q}(0,t)}, \widehat{q}(1,t)=0.
\end{array}\right.
\end{equation}
Consider system (\ref{erroerrorobserver}) in the space $\mathbb{H}=H^1_e(0,1)\times L^2(0,1)$, $H^1_e(0,1)=\{f\in H^1(0,1)|f(1)=0\}$ with inner product
induced norm $\|(f,g)\|^2_{\mathbb{H}}=\int_0^1[|f'(x)|^2+|g(x)|^2]dx$.
Define $\mathbb{A}(f,g)=(g,f''), \forall (f,g)\in D(\mathbb{A})=\{(f,g)\in (H_e^1\bigcap $
   $H^2(0,1))\times H^1_e(0,1)|f'(0)=\gamma g(0)+\beta f(0)\}.$
Then, system (\ref{erroerrorobserver}) can be written abstractly as
\begin{align*}
    \frac{d}{dt}(\widehat{q}(\cdot,t),\widehat{q}_t(\cdot,t))=\mathbb{A}(\widehat{q}(\cdot,t),\widehat{q}_t(\cdot,t)).
\end{align*}
By \cite{Chen1981}, $\mathbb{A}$ is generates an exponentially stable $C_0$-semigroup.
The following Lemma \ref{Riesz} tells us that the system (\ref{erroerrorobserver}) possess Riesz basis property.

\begin{lemma}\label{Riesz}
Assume that $\alpha\neq 1$. Then, $\mathbb{A}$ is a discrete operator;
there exist a sequence of generalized eigenfunctions of $\mathbb{A}$ which forms Riesz basis for $\mathbb{H}$; there exist a family of eigenvalue $\{\lambda_n\}_{n=-\infty}^\infty$ of operator $\mathbb{A}_2$ asymptotically expressed by
$\lambda_n=\frac{1}{2}\ln\bigg|\frac{\gamma-1}{\gamma+1}\bigg|+n_\gamma \pi i+O(|n|^{-1}),$
where $n_\gamma=n-1/2$ if $0<\gamma<1$ and $n_\gamma=n$ if $\gamma>1$;
the corresponding eigenfunctions is given by
$(\lambda_n^{-1}f_n,f_n)$ with $f_n(x)=\sinh\lambda_n(x-1)$ and
$F_n(x):=\big(\lambda_n^{-1}f_n'(x),f_n(x)\big)=\big(\big|\frac{\gamma-1}{\gamma+1}\big|^{(x-1)/2}e^{n_\gamma\pi i(x-1)}
+\big|\frac{\alpha+1}{\gamma-1}\big|^{(x-1)/2}e^{-n_\gamma\pi i(x-1)},  \big|\frac{\gamma-1}{\gamma+1}\big|^{(x-1)/2}e^{n_\gamma\pi i(x-1)}
-\big|\frac{\alpha+1}{\gamma-1}\big|^{(x-1)/2}$
$e^{-n_\gamma\pi i(x-1)} \big)
+O(|n|^{-1}).$
\end{lemma}

The following lemma is due to Feng and Guo \cite{Feng2017a}.
\begin{lemma}\label{zestimate}
Assume that $(\widehat{q}(\cdot,0),\widehat{z}_t(\cdot,0))\in D(\mathbb{A})$. The solution of (\ref{erroerrorobserver}) satisfies $|\widehat{q}_{x}(1,t)|\leq Me^{-\mu t}$ for some constant $M,\; \mu>0$.
\end{lemma}

\begin{remark} \em
On one hand, {\color{blue}since by (\ref{erroerrorobserver}) we obtain $\widehat{q}(1,t)=0$}, it follows from Lemma \ref{zestimate} that $F(t)=$
$-\widehat{v}_{x}(1,t)-m\widehat{v}_{tt}(1,t)
=-q_{x}(1,t)-mq_{tt}(1,t)+{\color{blue}\widehat{q}_{x}(1,t)}+m\widehat{q}_{tt}(1,t)=-q_{x}(1,t)-mq_{tt}(1,t)+{\color{blue}\widehat{q}_{x}(1,t)}\approx -q_{x}(1,t)-mq_{tt}(1,t)$ provided $(\widehat{q}(\cdot,0),\widehat{q}_t(\cdot,0))\in D(\mathbb{A})$.
This indicates that $-q_{x}(1,t)-mq_{tt}(1,t)$ is an estimate of total disturbance $F(t)$.
On the other hand, since (\ref{erroerrorobserver}) decays exponentially, we obtain  $u(\cdot,t)=v(\cdot,t)-q(\cdot,t)+\widehat{q}(\cdot,t)\approx v(\cdot,t)-q(\cdot,t)$. This implies that $v(\cdot,t)-q(\cdot,t)$ is an estimate of $w(\cdot,t)$.
The combination of the two hands tells us that the system (\ref{transfer}) estimates both the total disturbance and original state.
This is the reason why we call the system (\ref{transfer}) infinite-dimensional ESO for (\ref{beem}).
We shall show in the rest of the present section that the disturbance estimator (\ref{transfer}) is enough to exponentially stabilize system (\ref{beem}) for general initial state $(\widehat{q}(\cdot,0),\widehat{q}_t(\cdot,0))\in \mathbb{H}$.
\end{remark}

Since by \cite{Morgul1994} the state feedback $u(t)=-\alpha u_t(1,t)-a u_{xt}(1,t)$ exponentially stabilizes system (\ref{beem1}),
it is natural to design an estimated state-based controller
\begin{align}\label{feedback11}
    \nonumber&U(t)=q_{x}(1,t)+mq_{tt}(1,t)-\alpha[v_t(1,t)-q_t(1,t)]\\
    &-a [v_{xt}(1,t)-q_{xt}(1,t)],
\end{align}
where $q_{x}(1,t)+mq_{tt}(1,t)$ is used to cancel the total disturbance $F(t)$,  $v_t(1,t)-q_t(1,t)$ and $v_{xt}(1,t)-q_{xt}(1,t)$ are respectively applied to estimate $u_t(1,t)$ and $u_{xt}(1,t)$.
\begin{remark}\label{xiangdui}\em
In \cite{Mei2020b,Zhou2018a,Zhou2018b}, infinite-dimensional ESOs were used to estimate
the total disturbance only; the authors design additional ESO-based Luenberger state observers in order to derive the estimations of original states.
In our control law (\ref{feedback11}), we directly estimated state from the infinite-dimensional ESO.
Our control strategy is more concise and energy-saving, and it can help one simplify the design of references \cite{Mei2020b,Zhou2018a,Zhou2018b}.
\end{remark}

The closed-loop system of (\ref{beem}) under the controller (\ref{feedback11}) is
\begin{equation} \label{perror110}
\left\{\begin{array}{l}
u_{tt}(x,t)=u_{xx}(x,t),\;\; x\in (0,1), \; t>0, \\
u(0,t)=0,
u_{x}(1,t)+mu_{tt}(1,t)=q_{x}(1,t)+mq_{tt}(1,t)\\
-\alpha[v_t(1,t)
-q_t(1,t)]-a [v_{xt}(1,t)-q_{xt}(1,t)]
+F(t), \\
v_{tt}(x,t)=v_{xx}(x,t),\;\; x\in (0,1), \; t>0, \\
v_x(0,t)=\gamma v_{t}(0,t)+\beta v(0,t)+u_x(0,t), \; \;  t\ge 0,\\
v_{x}(1,t)+mv_{tt}(1,t)=q_{x}(1,t)+mq_{tt}(1,t)\\
-\alpha[v_t(1,t)-q_t(1,t)]
-a [v_{xt}(1,t)-q_{xt}(1,t)], \;\;  t\ge 0, \\
q_{tt}(x,t)=q_{xx}(x,t),\;\; x\in (0,1), \; t>0, \\
q_x(0,t)=\gamma q_{t}(0,t)+\beta q(0,t), \; \;  t\ge 0,\\
q(1,t)=v(1,t)-u(1,t), \;\;  t\ge 0,
\end{array}\right.
\end{equation}
which is equivalent to
\begin{equation} \label{perror110closed}
\left\{\begin{array}{l}
u_{tt}(x,t)=u_{xx}(x,t),\;\; x\in (0,1), \; t>0, \\
u(0,t)=0, u_{x}(1,t)+mu_{tt}(1,t)=-\alpha u_{t}(1,t)\\
-au_{xt}(1,t)+a \widehat{q}_{xt}(1,t)+\widehat{q}_x(1,t), \;\;  t\ge 0,\\
\widehat{v}_{tt}(x,t)=\widehat{v}_{xx}(x,t),\;\; x\in (0,1), \; t>0, \\
\widehat{v}_x(0,t)=\gamma \widehat{v}_t(0,t)+\beta \widehat{v}(0,t), \; \;   t\ge 0,\\
\widehat{v}_{x}(1,t)+m\widehat{v}_{tt}(1,t)=-F(t), \;\;  t\ge 0,\\
\widehat{q}_{tt}(x,t)=\widehat{q}_{xx}(x,t),\;\; x\in (0,1), \; t>0, \\
\widehat{q}_x(0,t)=\gamma\widehat{q}_t(0,t)+\beta\widehat{q}(0,t),
\widehat{q}(1,t)=0.
\end{array}\right.
\end{equation}

Our target in the rest of this section is to show that the state $(u(\cdot,t),u_t(\cdot,t))$ of the closed-loop system (\ref{perror110})
is exponentially stable while guaranteeing the boundedness of the other variables.
Since it does not involve the total disturbance $F(t)$,
we choose to firstly consider the $(w,\widehat{z})$-part of (\ref{perror110closed}) described by
\begin{equation} \label{wzwanclosed}
\left\{\begin{array}{l}
u_{tt}(x,t)=u_{xx}(x,t),\;\; x\in (0,1), \; t>0, \\
u(0,t)=0, u_{x}(1,t)+mu_{tt}(1,t)=-a u_{xt}(1,t)\\
-\alpha u_t(1,t)
+a \widehat{q}_{xt}(1,t)+\widehat{q}_x(1,t), \;\;  t\ge 0,\\
\widehat{q}_{tt}(x,t)=\widehat{q}_{xx}(x,t),\;\; x\in (0,1),  \\
\widehat{q}_x(0,t)=\gamma\widehat{q}_t(0,t)+\beta\widehat{q}(0,t),
\widehat{q}_x(1,t)=0.
\end{array}\right.
\end{equation}

We consider system (\ref{wzwanclosed}) in Hilbert state space $\mathcal{X}_1=\mathbf{H}$
$\times\mathbb{H}$.
Define the operator $\mathcal{A}_1:D(\mathcal{A}_1)\subset\mathcal{X}_1\rightarrow \mathcal{X}_1$ by
$\mathcal{A}_1(f,g,\eta,$
$\phi,\psi)=(g,f'',{\color{blue}-f'(1)-\alpha g(1)+\phi'(1)},\psi,\phi''),$
$(f,g,\phi,\psi)\in D(\mathcal{A}_1)=\{(f,g,\eta,\phi,\psi)\in (H^1_E(0,1)\bigcap $
$H^2(0,1))\times H^1_E(0,1)\times D(\mathbb{A})|
{\color{blue}\eta=a f'(1)+mg(1)-a \phi'(1)}\}$.
System (\ref{wzwanclosed}) is abstractly described by
$$\frac{d}{dt}X(t)=\mathcal{A}_1X(t),$$
where $X(t)=(u(\cdot,t),u_t(\cdot,t),a u_{x}(1,t)$
$+mu_t(1,t)-a\widehat{q}_{x}(1,t),\widehat{q}(\cdot,t),\widehat{q}_t(\cdot,t)).$
One can easily verify that the operator $\mathcal{A}_1$ is not dissipative,
 we shall adopt Riesz basis approach to verify the stability.

\begin{theorem}\label{exponential00}
Assume that $\gamma\neq 1$, $m\neq a$ and $m\neq a\gamma$. Then, system (\ref{wzwanclosed}) is governed by an exponentially stable $C_0$-semigroup.
\end{theorem}

\begin{proof}\ \
By the same procedure as the proof of Theorem \ref{exponentialnodisturbance}, we can obtain that the operator $\mathcal{A}_1$ has bounded and compact inverse on $\mathcal{X}_1$ and $\sigma(\mathcal{A}_1)=\sigma(\mathbf{A})\bigcup \sigma(\mathbb{A})$.
Next, we shall show that the generalized eigenfunction of $\mathcal{A}_1$ forms Riesz basis for $\mathcal{X}_1$.
Let $\{\lambda_{n}\}_{n=-\infty}^\infty$ and $\{\lambda_{1n}\}_{n=-\infty}^\infty$ be the eigenvalues of $\mathbb{A}$ and $\mathbf{A}$,  respectively. Let $\{(\lambda_n^{-1}\phi_{n},\phi_{n})\}_{n=-\infty}^\infty$ and $\{(\lambda_{1n}^{-1}f_{n},f_{n},a f_n'(1)+m\lambda_{1n}f_n(1))\}_{n=-\infty}^\infty$ be the generalized eigenfunctions corresponding to $\{\lambda_{n}\}_{n=1}^\infty$ and $\{\lambda_{1n}\}_{n=-\infty}^\infty$ that form Riesz basises for $\mathbb{H}$ and $ \mathbf{H}$, respectively. As a result, $\{(0,0,0,\lambda_n^{-1}\phi_{n},\phi_{n})\}_{n=-\infty}^\infty\bigcup \{(\lambda_{1n}^{-1}f_{n},f_{n},$
$a f_n'(1)+m\lambda_{1n}f_n(1),0,0)\}_{n=-\infty}^\infty$ forms a Riesz basis for $\mathcal{X}_1$, which is equivalent to that $\{(0,0,0,\lambda_n^{-1}\phi'_{n},$
$\phi_{n})\}_{n=-\infty}^\infty\bigcup \{(\lambda_{1n}^{-1}$
$f'_{n},f_{n},a f_n'(1)+m\lambda_{1n}f_n(1),0,0)\}_{n=-\infty}^\infty$  forms a Riesz basis for $(L^2(0,1))^2\times \mathds{C}\times (L^2(0,1))^2$.

Let $\lambda\in \sigma(\mathcal{A})$ and $(\lambda^{-1}f,f,a f'(1)+m\lambda f(1),\lambda^{-1}\phi,\phi)$ be the corresponding eigenfunction. If $\phi=0$, then $(\lambda^{-1}f,f,a f'(1)+m\lambda f(1))\neq 0$ and $\lambda\in \sigma(\mathbf{A})$.
Hence in this case the eigenvalues $\{\lambda_{1n}\}_{n=1}^\infty$ corresponds the eigenfunctions $\{(\lambda_{1n}^{-1}f_n,f_n,a f'_n(1)+m\lambda_{1n} f_n(1),0,0)\}_{n=1}^\infty$.

If $\phi\neq 0$, then $\lambda\in \sigma(\mathbb{A})$. The eigenvalues $\{\lambda_{n}\}_{n=1}^\infty$ corresponds
the eigenfunction $(\lambda_n^{-1}\phi_{n},\phi_{n})$
$\}_{n=1}^\infty$ of $\mathbb{A}$, where
$\phi_{n}=\sinh\lambda_n(x-1).$
Denote by $(\lambda_n^{-1}f_{1n},f_{1n},a f'_{1n}(1)-a \phi'_{n}(1)+m\lambda_{n} f_{1n}(1),$
$\lambda_n^{-1}\phi_{n},\phi_{n})$ the eigenfunction of $\mathcal{A}_1$ corresponding to $\lambda_{n}$. Then we have
$f''_{1n}(x)=\lambda_nf_{1n}(x),
        f_{1n}(0)=0,
       (1+a\lambda_n)f'_{1n}(1)+(m\lambda^2_n+\alpha\lambda_n)f_{1n}=(a\lambda_n+1)\phi'_n(1)$
whose solution is of the form
\begin{align*}
    f_{1n}(x)=F\sinh\lambda_n(x-1)+G\cosh\lambda_n(x-1).
\end{align*}
Combine (\ref{f1n}) and the boundary conditions to derive
$F=[(1+a\lambda_n)\cosh\lambda_n]/h_n,
    G=[(1+a\lambda_n)\sinh\lambda_n]/h_n,$ where $h_n=(1+a\lambda_n)\cosh\lambda_n+(m\lambda_n+\alpha)\sinh\lambda_n.$
Use the fact that $\lambda_n\cosh\lambda_n=-(\gamma\lambda_n+\beta)\sinh\lambda_n$ to get
$F=\frac{a\gamma}{a\gamma -m}+O(|n|^{-1})$
and $G=-\frac{a}{a\gamma -m}+O(|n|^{-1}).$
Denote $Q=\left(
            \begin{array}{cc}
              I_3 & J \\
              0 & I_2 \\
            \end{array}
          \right)
$ with $J=\frac{a}{a\gamma-m}\left(
                              \begin{array}{ccccc}
                                \gamma & -1& 0 \\
                                 -1 &\gamma& 0  \\
                                 0 &0&0 \\
                              \end{array}
                            \right)^T.$
Then $Q$ is a bounded linear operator and it has bounded inverse.
Moreover, we obtain the following relations
\begin{align}
\nonumber    &(\lambda_{1n}^{-1}f'_n,f_n,a f'_n(1)+m\lambda_{1n} f_n(1),0,0)^T\\
\label{guanxi3}&=Q(\lambda_{1n}^{-1}f'_n,f_n,a f'_n(1)+m\lambda_{1n} f_n(1),0,0)^T,\\
 \nonumber    &(\lambda_n^{-1}f'_{1n},f_{1n},a f'_{1n}(1)-a \phi'_{n}(1)+m\lambda_{n} f_{1n}(1),\lambda_n^{-1}\phi'_{n},\\
\label{guanxi4}&\phi_{n})=Q(0,0,0,\lambda_n^{-1}\phi'_{n},\phi_{n})+O(n^{-1}).
\end{align}
Then, by Bari's theorem the sequence $\{(\lambda_{1n}^{-1}f'_n,f_n,$
$a f'_n(1)+m\lambda_{1n} f_n(1),0,0)\}_{n=-\infty}^\infty \bigcup \{(\lambda_n^{-1}f'_{1n},f_{1n},$
$a f'_{1n}(1)-a \phi'_{n}(1)+m\lambda_{n} f_{1n}(1),\lambda_n^{-1}\phi'_{n},\phi_{n})\}_{n=-\infty}^\infty$
forms Riesz basis for $\left(L^2(0,1)\right)^2\times \mathds{C}\times \left(L^2(0,1)\right)^2$,
which is equivalent to that $\{(\lambda_{1n}^{-1}f_n,f_n,a f'_n(1)+m\lambda_{1n} f_n(1),0,0)\}_{n=-\infty}^\infty \bigcup$
$ \{(\lambda_n^{-1}f'_{1n},f_{1n},a f_{1n}(1)-a \phi'_{n}(1)+m\lambda_{n} f_{1n}(1),\lambda_n^{-1}\phi_{n},\phi_{n})$
$\}_{n=-\infty}^\infty$
forms Riesz basis for $\mathcal{X}_1$.

The semigroup generation and spectrum-determined growth condition of $\mathcal{A}_1$ are directly derived by the Riesz basis property.
Since $ \mathbf{A}$ and $\mathbb{A}$ generate exponentially stable $C_0$-semigroups and spectrum-determined growth condition holds, $e^{\mathcal{A}_1t}$ is exponentially stable.
\end{proof}

\begin{remark} \em
In the proof of Theorem \ref{exponential00}, although it forms a Riesz basis
for $\mathcal{H}_1$,
{\color{blue} not all the elements of the sequence $\{(0,0,0,\lambda_n^{-1}\phi_{n},$
$\phi_{n})\}_{n=-\infty}^\infty\bigcup \{(\lambda_{1n}^{-1}f_{n},f_{n},a f_n'(1)+m\lambda_{1n}f_n(1),0,0)\}_{n=-\infty}^\infty$  are generalized eigenfunctions of $\mathcal{A}$}.
In order to overcome this difficulty,  our key step is to find out the relations (\ref{guanxi3}) and (\ref{guanxi4}).
\end{remark}

We consider the closed-loop system (\ref{perror110}) in Hilbert state space $\mathfrak{X}=\mathbf{H}_1\times \mathbf{H}_2\times \mathbb{H}_1$,
where $\mathbb{H}_1=H^1(0,1)\times L^1(0,1)$ with norm {\color{blue}$\|(f,g)\|^2_{\mathbb{H}_1}=\int_0^1[|f'(x)|^2+|g(x)|^2]dx+\beta|f(0)|^2$}.
The state of the closed-loop system (\ref{perror110}) is $(u(\cdot,t),u_t(\cdot,t),mu_t(1,t)-mq_t(1,t)+a(v_x(1,t)-q_x(1,t)),v(\cdot,t),v_t(\cdot,t),
mv_t(1,t)-mq_t(1,t)+a[v_x(1,t)-q_x(1,t)],q(\cdot,t),q_t(\cdot,t))$.

\begin{theorem}\label{maintheorem}
Suppose that $d\in L^\infty(0,\infty)$ (or $d\in L^2(0,\infty)$) and $f:H^1(0,1)\times L^2(0,1)\rightarrow \mathds{R}$ is continuous.
For any initial value $(u(\cdot,0),u_t(\cdot,0),mu_t(1,0)-mq_t(1,0)+a(v_x(1,0)-q_x(1,0)),v(\cdot,0),v_t(\cdot,$
$0),mv_t(1,0)-mq_t(1,0)+a[v_x(1,0)-q_x(1,0)],q(\cdot,0),q_t(\cdot,0)) \in \mathfrak{X}$ with $q(1,0)=v(1,0)-u(1,0)$, then there exists a unique solution to system (\ref{perror110}) such that
$(u(\cdot,t),u_t(\cdot,t),mu_t(1,t)$
$-mq_t(1,t)+a(v_x(1,t)-q_x(1,t)),v(\cdot,t),v_t(\cdot,t),
mv_t(1,t)-mq_t(1,t)+a[v_x(1,t)-q_x(1,t)],q(\cdot,t),q_t(\cdot,t)) \in   C(0,\infty;$
$\mathfrak{X})$  satisfy $q(1,t)=v(1,t)-u(1,t)$,
\begin{align}
  \nonumber&\int_0^1[|u_t(x,t)|^2+|u_{x}(x,t)|^2]{\color{blue} dx}+\frac{1}{m}|mu_t(1,t)\\
 \label{P1}&+a[v_x(1,t)-q_x(1,t)]|^2\leq M_2e^{-\gamma_2 t},\\
  \nonumber&\sup_{t\geq 0}\int_0^1[|u_t(x,t)|^2+|u_{x}(x,t)|^2]{\color{blue} dx}+\frac{1}{m}|mu_t(1,t)\\
 \label{P0} &-mq_t(1,t)+a[v_x(1,t)-q_x(1,t)]|^2<\infty,\\
   \nonumber &\sup_{t\geq 0} \bigg[\int_0^1[|v_t(x,t)|^2+|v_{x}(x,t)|^2+|q_t(x,t)|^2+|q_{x}(x,\\
   \nonumber&t)|^2]dx+\beta(|v(0,t)|^2+|q(0,t)|^2)+|mv_t(1,t)-mq_t(1,\\
\label{P2}&t)+a[v_x(1,t)-q_x(1,t)]|^2)\bigg]<\infty,
\end{align}
where $M_2$ and $\gamma_2$ are two positive constants.

If $f(0,0)=0$ and $d\in L^2(0,\infty)$, then
$\lim_{t\rightarrow +\infty} \int_0^1(|u_t(x,$
$t)|^2+|u_{x}(x,t)|^2+\frac{1}{m}|mu_t(1,t)-$
$mq_t(1,t)+a(v_x(1,t)-q_x(1,t))|^2=0,
\lim_{t\rightarrow +\infty} \big[\int_0^1(|v_t(x,$
$t)|^2+|v_{x}(x,t)|^2+|q_t(x,t)|^2
+|q_{x}(x,t)|^2)dx+\beta(v(0,t)|^2$
$+|q(0,t)|^2)+|mv_t(1,t)
-mq_t(1,t)+a[v_x(1,t)-q_x(1,t)|^2)\big]$
$=0.$
\end{theorem}
\begin{proof}\ \
Given initial value $(u(\cdot,0),u_t(\cdot,0),mu_t(1,0)$
$-mq_t(1,0)+a(v_x(1,0)-q_x(1,0)),v(\cdot,0),v_t(\cdot,0),$
$mv_t(1,0)-mq_t(1,0)+a[v_x(1,0)-q_x(1,0)],q(\cdot,0),q_t(\cdot,0)) \in \mathbf{H}_1\times \mathbf{H}_2\times \mathbb{H}_1$ with $q(1,0)=v(1,0)-u(1,0)$, we have $(u(\cdot,0),u_t(\cdot,0),mu_t(1,0)+a[u_t(1,0)-\widehat{q}_t(1,0)],\widehat{q}(1,0),$
$\widehat{q}_t(1,0))\in \mathcal{X}_1$ and $\big(\widehat{v}(\cdot,0),\widehat{v}_t(\cdot,$
$0),m\widehat{v}_t(1,0)\big)\in \mathbf{H}_2$.
By Theorem \ref{exponential00}, we obtain $(u(\cdot,t),u_t(\cdot,t),mu_t(1,t)+a[u_t(1,t)-\widehat{q}_t(1,t)],\widehat{q}(1,t),$
$\widehat{q}_t(1,t))\in C(0,\infty;\mathcal{X}_1), q(1,t)$
$=v(1,t)-u(1,t)$ and
\begin{align*}
    &\int_0^1(|u_t(x,t)|^2+|u_{x}(x,t)|^2)dx+|mu_t(1,t)+a(v_x(1,t)\\
  &-q_x(1,t))|^2=\int_0^1(|u_t(x,t)|^2+|u_{x}(x,t)|^2)dx\\
  &+|mu_t(1,t)+a(u_x(1,t)-\widehat{q}_x(1,t))|^2\\
&\leq M^2_{\mathcal{A}_1}e^{-2\omega_{\mathcal{A}_1}t}\|(u(\cdot,0),u_t(\cdot,0),mu_t(1,0)\\
&+a(u_x(1,0)-\widehat{q}_x(1,0)),\widehat{q}(\cdot,0),\widehat{q}_t(\cdot,0)) \|^2_{\mathcal{X}_1}= M^2e^{-\gamma_2 t},
\end{align*}
where $\gamma_2=2\omega_{\mathcal{A}_1}$ and $M_2=M^2_{\mathcal{A}_1}
\|(u(\cdot,0),u_t(\cdot,0),mu_t(1,$
$0)+a(u_x(1,0)-\widehat{q}_x(1,0)),\widehat{q}(\cdot,0),\widehat{q}_t(\cdot,0)) \|^2_{\mathcal{X}_1}$
$=M^2_{\mathcal{A}_1}\|(u(\cdot,$
$0),u_t(\cdot,0),mu_t(1,0)+a(v_x(1,0)-q_x(1,0)),q(\cdot,0)-v(\cdot,0)+u(\cdot,0),q_t(\cdot,0)-v_t(\cdot,0)+u_t(\cdot,0)) \|^2_{\mathcal{X}_1}.$

Therefore, we derive the continuity and exponential stability of $(u(\cdot,t),u_t(\cdot,t))$ on $H^1(0,1)\times L^2(0,1)$.
This, together with the continuity of $f$ indicates that $f(w(\cdot,t),w_t(\cdot,t))\in L^\infty(0,\infty)$.
It follows from Lemma \ref{admissible} that $\big(\widehat{v}(\cdot,t),\widehat{v}_t(\cdot,t),m\widehat{v}_t(1,t)\big)\in C(0,\infty;\mathbf{H}_2)$
and $\sup_{t\geq 0}\|\big(\widehat{v}(\cdot,t),\widehat{v}_t(\cdot,t),m\widehat{v}_t(1,$
$t)\big)\|_{\mathbf{H}_2}<+\infty$. Hence,
$mu_t(1,t)-mq_t(1,t)+a(v_x(1,t)-q_x(1,t))
=mu_t(1,t)$
$+a[u_t(1,t)-q_t(1,t)]-m\widehat{v}_t(1,t)\in C(0,\infty,\mathds{C}),$
 $mv_t(1,t)$
 $-mq_t(1,t)+a(v_x(1,t)-q_x(1,t))=mu_t(1,t)+a(v_x(1,t)-q_x(1,t))\in C(0,\infty,\mathds{C}),$
 $\sup_{t\geq 0}|mu_t(1,t)-$
 $mq_t(1,t)+a(v_x(1,t)-q_x(1,t))|\leq \sup_{t\geq 0}|mu_t(1,t)+a[u_t(1,t)-\widehat{q}_t(1,t)]|$
 $+\sup_{t\geq 0}|m\widehat{v}_t(1,t)|<\infty,$
  $\sup_{t\geq 0}| m$
  $v_t(1,t)-mq_t(1,t)+a(v_x(1,t)-q_x(1,t))|=\sup_{t\geq 0}|mu_t(1,$
  $t)+a(v_x(1,t)-q_x(1,t))|<\infty,$
\begin{align*}
   & \int_0^1[|v_t(x,t)|^2+|v_{x}(x,t)|^2]dx+\beta|v(0,t)|^2\leq 2\int_0^1[|\widehat{v}_t(x,\\
&t)|^2+|\widehat{v}_{x}(x,t)|^2]dx+\beta|\widehat{v}(0,t)|^2+2\int_0^1[|u_t(x,t)|^2\\
&+|u_{x}(x,t)|^2]dx\leq 2\sup_{t\geq 0}\bigg[\|\big(\widehat{v}(\cdot,t),\widehat{v}_t(\cdot,t),m\widehat{v}_t(1,t)\big)\|^2_{\mathbf{H}_2}
\\
&+2\|(u(\cdot,t),u_t(\cdot,t),mu_t(1,t)+a(u_x(1,t)-\widehat{q}_x(1,t)),\\
&\widehat{q}(\cdot,t),\widehat{q}_t(\cdot,t)) \|^2_{\mathcal{X}_1}\bigg]<\infty,\\
   & \int_0^1[|q_t(x,t)|^2+|q_{x}(x,t)|^2]dx+\beta|q(0,t)|^2\leq 2\int_0^1[|\widehat{q}_t(x,\\
&t)|^2+|\widehat{q}_{x}(x,t)|^2]dx+2\beta|\widehat{q}(0,t)|^2+2\int_0^1[|\widehat{v}_t(x,t)|^2\\
&+|\widehat{v}_{xx}(x,t)|^2]dx+2\beta|\widehat{v}(0,t)|^2\leq 2\int_0^1[|\widehat{q}_t(x,t)|^2\\
&+|\widehat{q}_{x}(x,t)|^2]dx+2\beta\int_0^1|\widehat{q}_{x}(x,t)|^2dx+2\int_0^1[|\widehat{v}_t(x,t)|^2\\
&+|\widehat{v}_{xx}(x,t)|^2]dx+2\beta|\widehat{v}(0,t)|^2\leq 2(1\\
&+\beta)M^2_{\mathbb{A}}e^{-2\omega_\mathbb{A}}\|(\widehat{q}(\cdot,t),\widehat{q}_t(\cdot,t))\|^2_{\mathbb{H}}+2\int_0^1[|\widehat{v}_t(x,t)|^2\\
&+|\widehat{v}_{xx}(x,t)|^2]dx+2\beta|\widehat{v}(0,t)|^2<\infty,
\end{align*}
where the fact $|\widehat{q}(0,t)|^2=\big|\widehat{q}(1,t)-\int_0^1\widehat{q}_x(x,t)dx\big|^2\leq \int_0^1|\widehat{q}_{x}(x,t)|^2dx$ is used.
Then (\ref{P0}) is derived, system (\ref{perror110}) has a unique solution and $(u(\cdot,t),u_t(\cdot,t),mu_t(1,t)-mq_t(1,t)+a(v_x(1,t)-q_x(1,t)),v(\cdot,t),v_t(\cdot,t),
mv_t(1,t)-mq_t(1,t)+a[v_x(1,t)-q_x(1,t)],q(\cdot,t),q_t(\cdot,t)) \in   C(0,\infty;$
$\mathfrak{X})$ and (\ref{P2}) is derived.

If $f(0,0)=0$, we derive $\lim_{t\rightarrow \infty}f(w(\cdot,t),w_t(\cdot,t))=0$ by the continuity of $f$ and the exponential stability of $(w(\cdot,t),w_t(\cdot,t))$. Moreover, we use the assumption $d\in L^2(0,\infty)$ and \cite[Lemma A.1]{Zhou2018a} to get
$\lim_{t\rightarrow\infty}\|(\widehat{v}(\cdot,t),\widehat{v}_t(\cdot,t),m\widehat{v}_t(1,t))\|_{\mathbb{H}_1}=0.$
Then, the two limits are verified. This completes the proof.
\end{proof}

\begin{remark} \em
The subsystem consisting of $(u,\widehat{q})$-part of
(\ref{wzwanclosed}) is vital to prove Theorem \ref{maintheorem}.
Indeed, by virtue of semigroup theorem, we can derive the continuity and exponential stability of the solution of (\ref{beem})
without the global Lipsichitz condition. Then we obtain that $(u(\cdot,t),u_t(\cdot,t))$ is continuous and bounded on $H^1_E(0,1)\times L^2(0,1)$.
By Lemma \ref{admissible}, the existence and boundedness of $(\widehat{v}(\cdot,t),\widehat{v}_t(\cdot,t),m\widehat{v}_t(1,t))$ are verified by viewing $f(w(\cdot,t),w_t(\cdot,t))+d(t)$ as the boundary input, provided $f$ is continuous. Furthermore,
we not only derive the exponential stability of $(u(\cdot,t),u_t(\cdot,t))$ but also verify the exponential stability of
$\eta(t)=mu_t(1,t)+a(v_t(1,t)-q_t(1,t))$, because $mu_t(1,t)+a(v_t(1,t)-q_t(1,t))$ is the state of boundary dynamic of the exponentially stable system (\ref{perror110closed}). Theorem \ref{maintheorem} also tells us that the other states (including the state of the boundary dynamic) are bounded.
\end{remark}

\begin{remark} \em
In Theorem \ref{maintheorem}, compared to the results in \cite{Xie2017}, the differences and improvements mainly lie in that,
1) the internal uncertainty is taken into consideration, while \cite{Xie2017} just studies the case of $f(w,w_t)=0$,
2) only ``low order'' measurements $w(1,t)$ and $w_{x}(0,t)$ are adopted, while \cite{Xie2017} used the velocity $w_{t}(1,t)$ as well as the high order
angular velocity $w_{xt}(1,t)$, 3) we consider arbitrary $m\neq a,m\neq a\gamma,\gamma\neq1$ while \cite{Xie2017} just solved the special case $m=a\alpha$.
\end{remark}

\section{Numerical simulation}\label{shiyan}

We present in this section some numerical simulations for the closed-loop system (\ref{perror110}).
We use the finite difference scheme and the numerical results are programmed in Matlab.
The time step and the space step are taken as $1/200$ and $1/100$, respectively. We take the nonlinear internal uncertainty $f(u(\cdot,t),u_t(\cdot,t)) = \sin(u(1,t))$ and
the external disturbance $d(t)=\cos(2t)$. Let $m=5$. The parameters and the initial values are chosen as $\alpha=a=2,\beta=\gamma=1.5,
u(x,0)=x^3-3x^2, u_t(x,0)=0, v(x,0)=-2x^3, v_t(x,0)=0, q(x,0)=q_t(x,0)=0.$

Fig.1, Fig.3 and Fig.5 show the displacements $u(x,t)$, $v(x,t)$ and $q(x,t)$ of the closed-loop system (\ref{perror110}), respectively;
while Fig.2, Fig 4 and Fig 6 respectively present the velocities $u_t(x,t)$, $v_t(x,t)$ and $q_t(x,t)$.
Fig.7 and Fig.8 display the boundary sates $\eta(t)=mu_t(1,t)+a[v_x(1,t)-q_x(1,t)]$ and $\psi(t)=mu_t(1,t)-mq_t(1,t)+a[v_x(1,t)-q_x(1,t)]$.
One can see from the simulations that $(u(\cdot,t),u_t(\cdot,t))$ and $\eta(t)$ decays rapidly, while $(v(\cdot,t),v_t(\cdot,t))$, $(q(\cdot,t),q_t(\cdot,t))$ and $\psi(t)$ are bounded.

\begin{figure}
\begin{minipage}[t]{0.45\linewidth}
\centering     
\includegraphics[height=5.5cm,width=7.2cm]{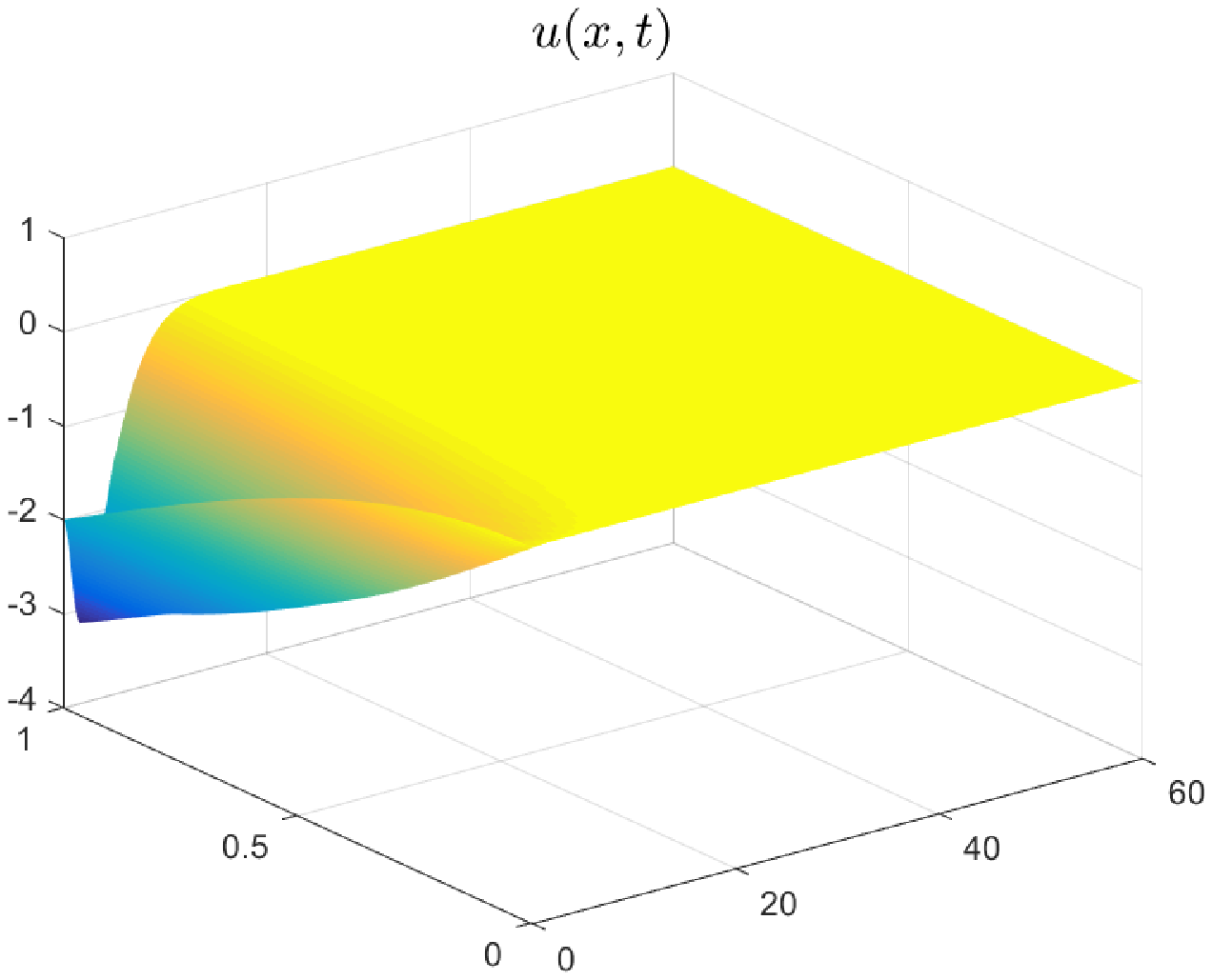}
\caption{The state $u(x,t)$.}
\label{1}
\end{minipage}
\hfill
\begin{minipage}[t]{0.45\linewidth}
\centering
\includegraphics[height=5.5cm,width=7.2cm]{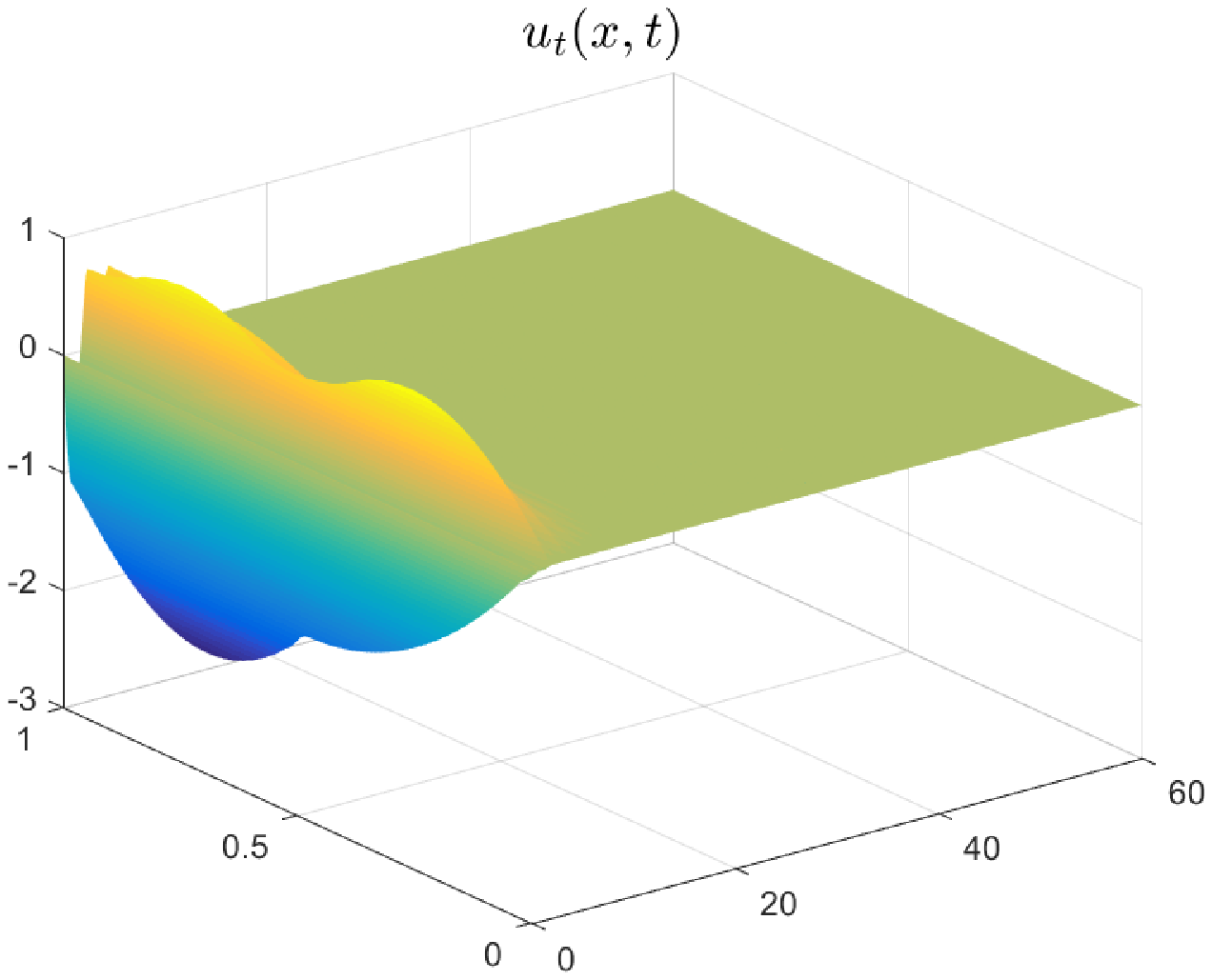}
\caption{The state $u_t(x,t)$.}
\label{2}
\end{minipage}
\end{figure}

\begin{figure}
\begin{minipage}[t]{0.45\linewidth}
\centering     
\includegraphics[height=5.5cm,width=7.2cm]{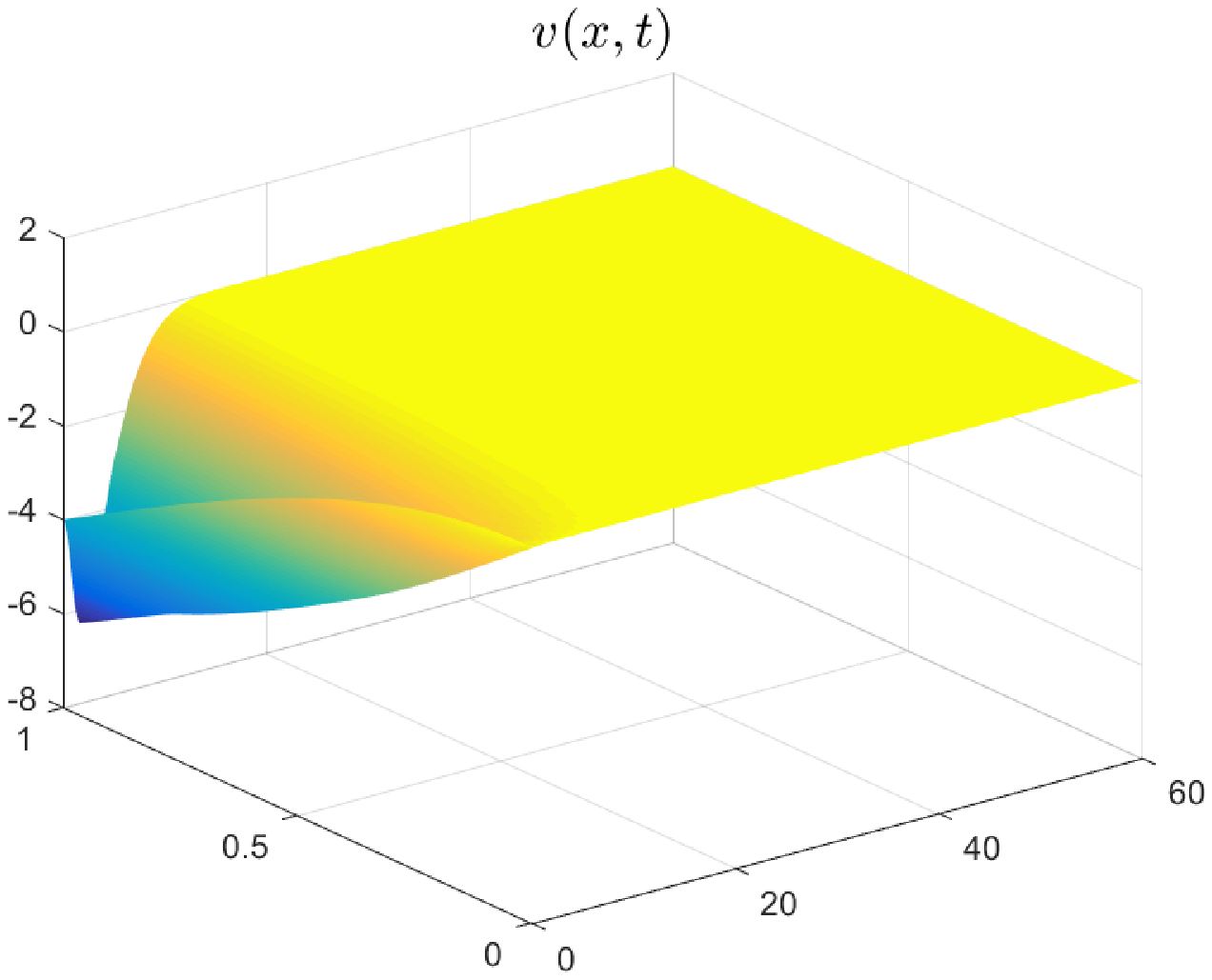}
\caption{The state $v(x,t)$.}
\label{3}
\end{minipage}
\hfill
\begin{minipage}[t]{0.45\linewidth}
\centering
\includegraphics[height=5.5cm,width=7.2cm]{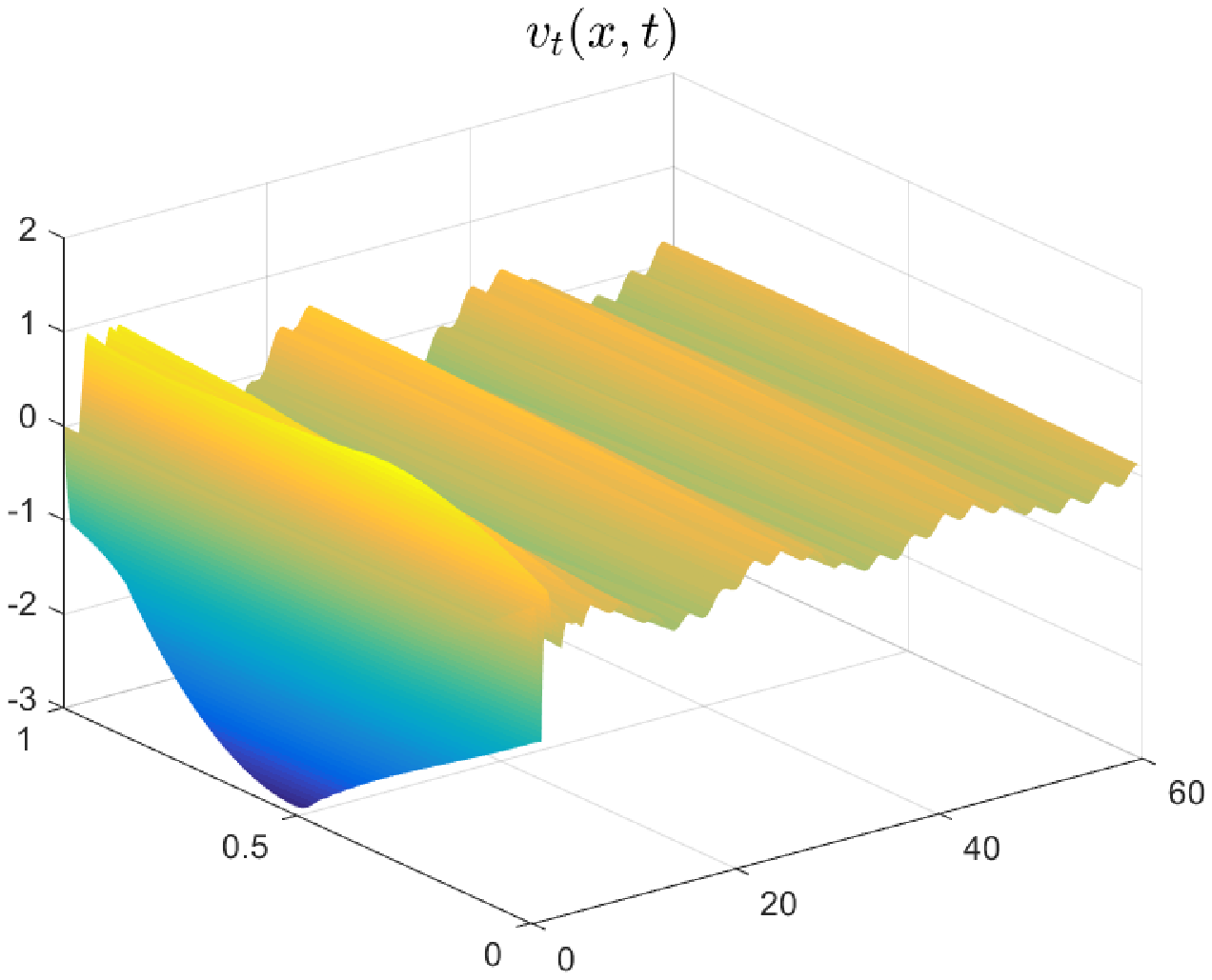}
\caption{The state $v_t(x,t)$.}
\label{4}
\end{minipage}
\end{figure}

\begin{figure}
\begin{minipage}[t]{0.45\linewidth}
\centering     
\includegraphics[height=5.5cm,width=7.2cm]{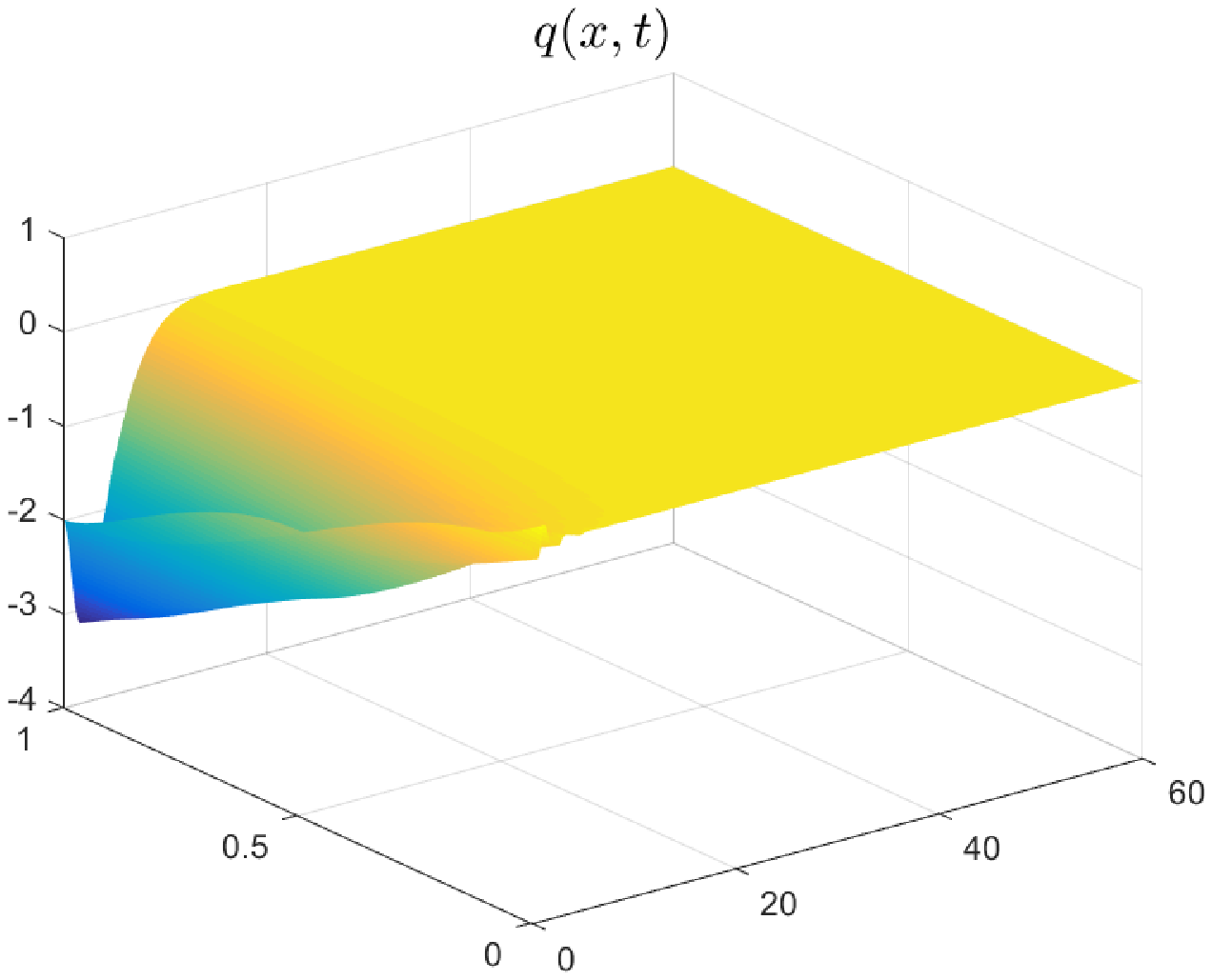}
\caption{The state $q(x,t)$.}
\label{5}
\end{minipage}
\hfill
\begin{minipage}[t]{0.45\linewidth}
\centering
\includegraphics[height=5.5cm,width=7.2cm]{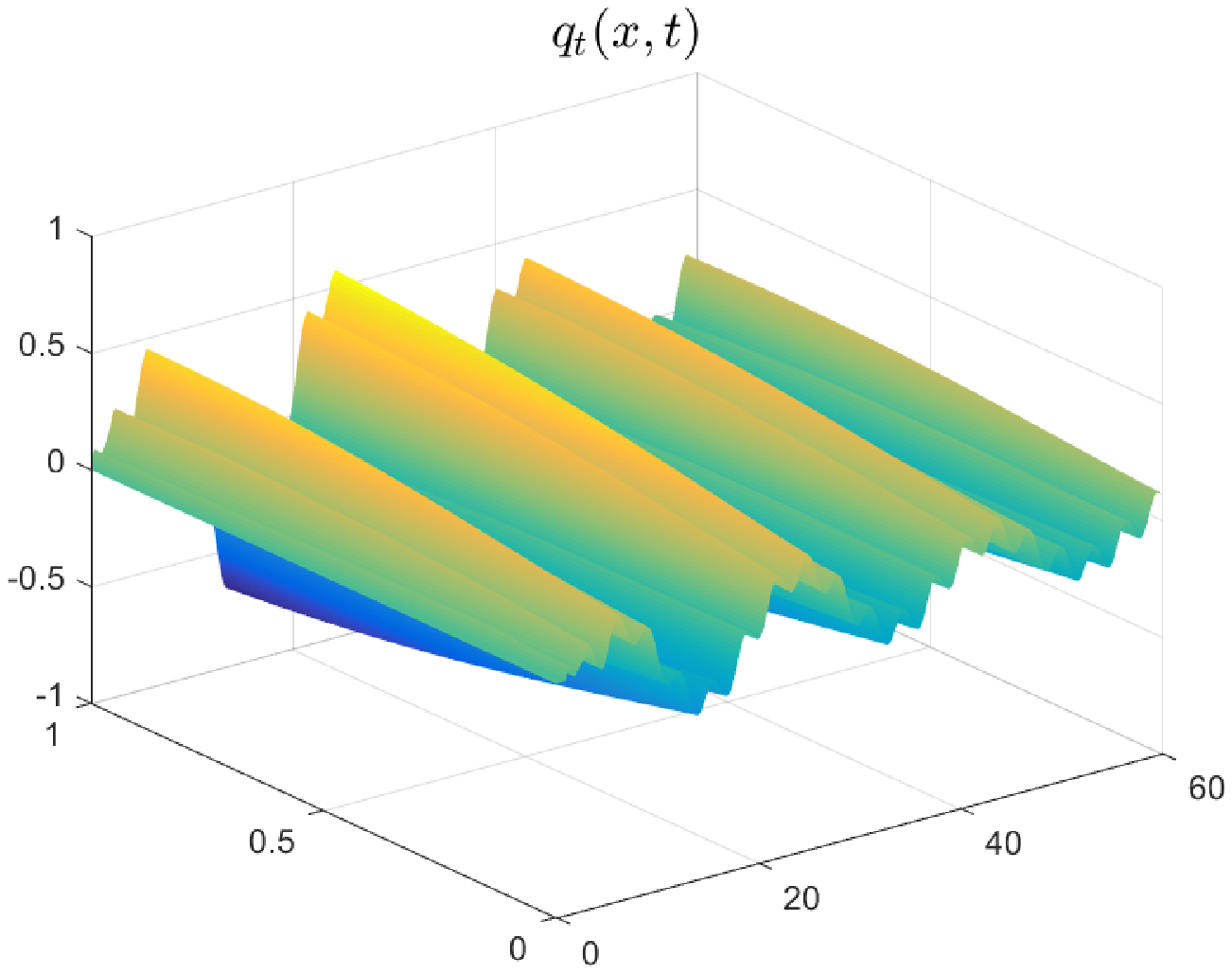}
\caption{The state $q_t(x,t)$.}
\label{6}
\end{minipage}
\end{figure}

\begin{figure}
\begin{minipage}[t]{0.45\linewidth}
\centering     
\includegraphics[height=5.5cm,width=7.2cm]{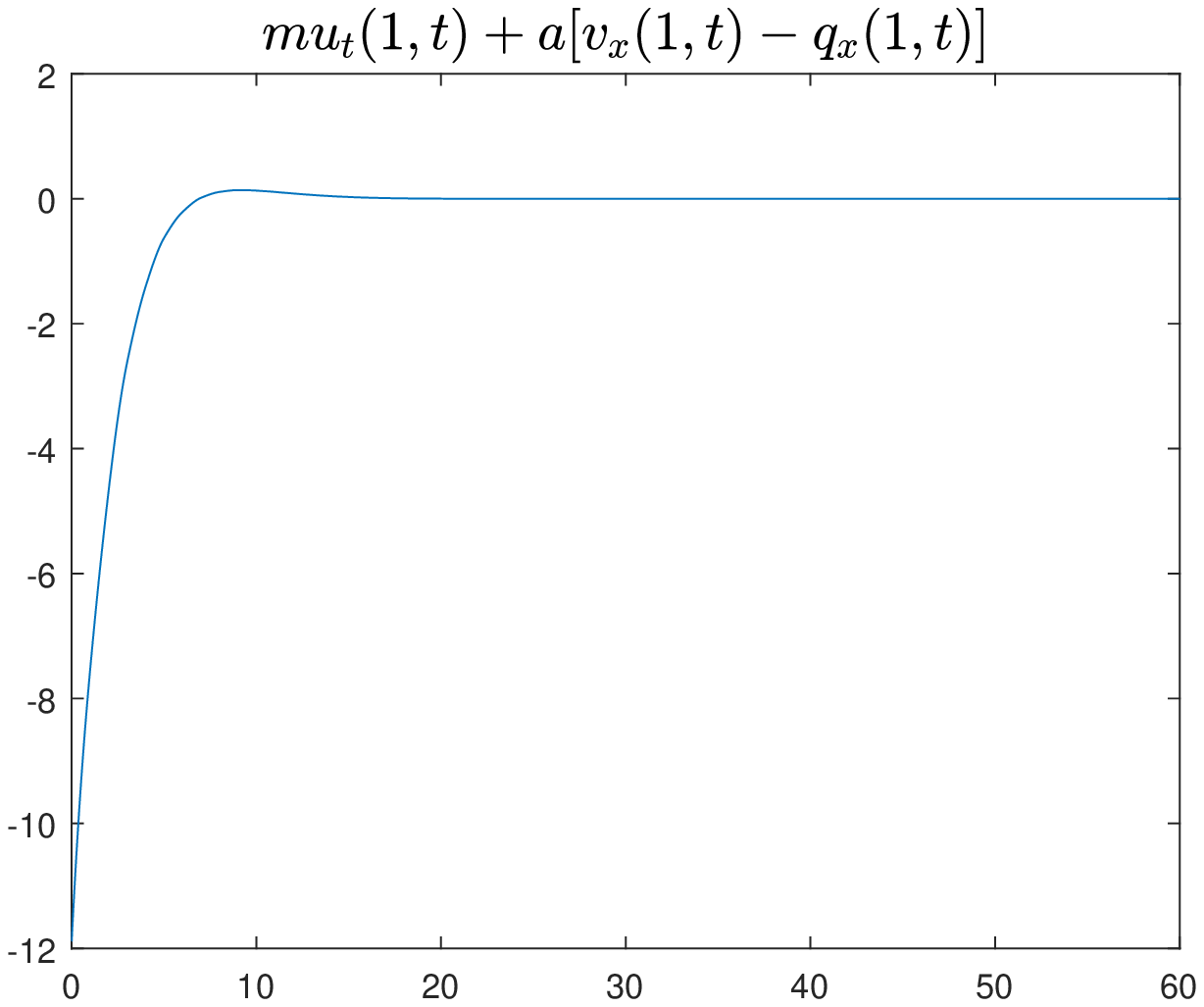}
\caption{The state $\eta(x,t)$}
\label{7}
\end{minipage}
\hfill
\begin{minipage}[t]{0.45\linewidth}
\centering
\includegraphics[height=5.5cm,width=7.2cm]{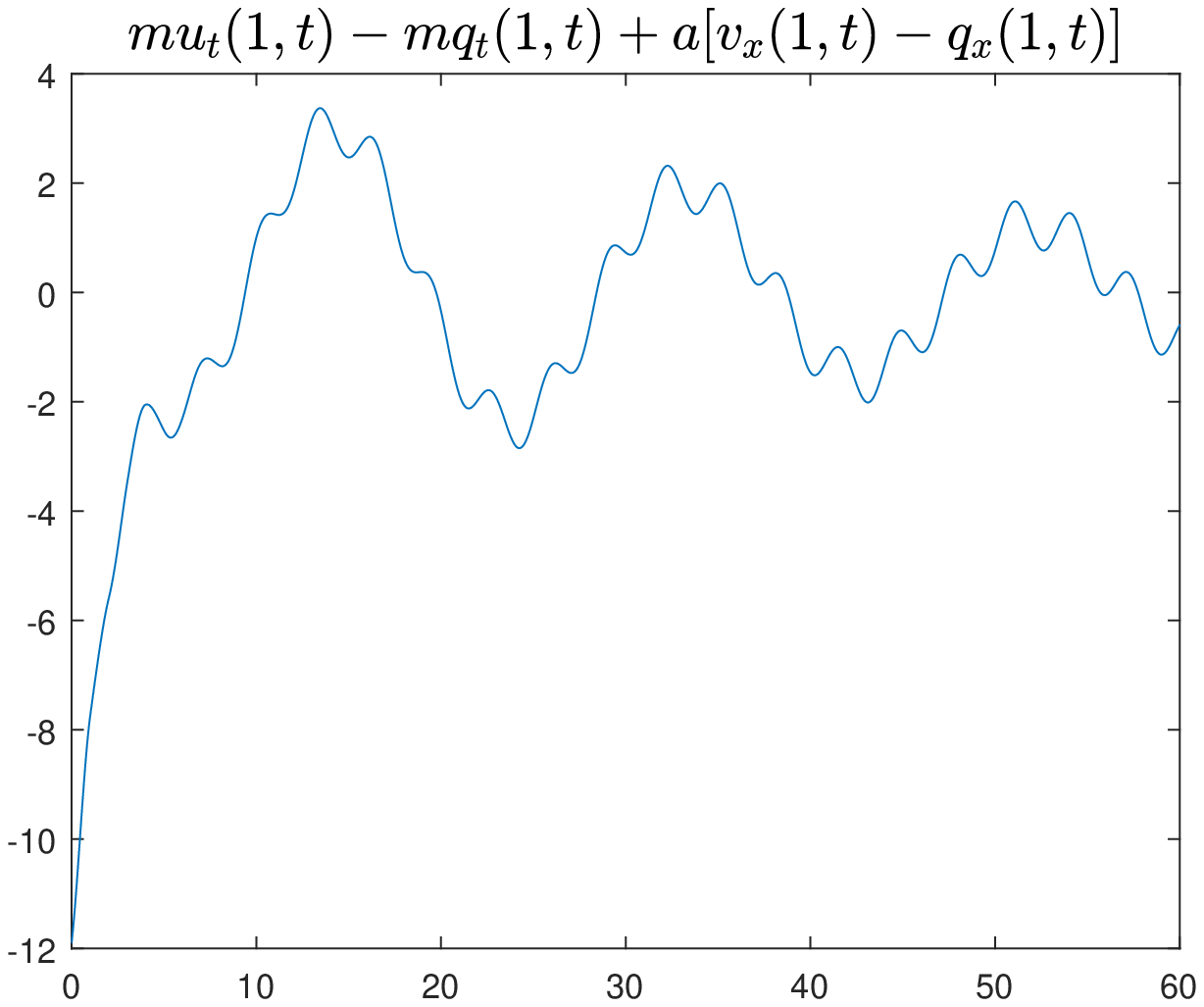}
\caption{The state $\psi(x,t)$}
\label{8}
\end{minipage}
\end{figure}

\section{Concluding  remarks}

In this paper, the output feedback exponential stabilization for a 1-d wave PDE with boundary and with or without disturbance is investigated.
When there is no disturbance, with only one non-collocated measurement $w_x(0,t)$, we design a Luenberger state observer and an estimated state based stabilizing controller. This improves the existence references \cite{Guo2000,Morgul1994} where the authors used two collocated measurements $w_t(1,t)$ and $w_{xt}(1,t)$.
By modifying the proof, we can simplify the proof of the exponential stability of the closed-loop system \cite[(3.1)]{Guo2007} where $m=0$,
because our coupled system (\ref{closednodisturbance1}) contains an independent subsystem.
When the boundary internal uncertainty and external disturbance are considered,
we construct an infinite-dimensional ESO to estimate the original state and total disturbance online.
Then, the estimated state and estimated total disturbance
allows us to design a stabilizing controller while guaranteeing the boundedness of the closed-loop system.
Riesz basis approach is the main tool for the proofs of the exponential stabilities of two coupled systems.

\end{document}